\UseRawInputEncoding
\documentclass{article}%
\usepackage{amsfonts}%
\usepackage{amsmath}%
\usepackage{amsthm}%
\usepackage{amssymb}%
\usepackage{MnSymbol}
\usepackage{graphicx}
\usepackage{accents}
\usepackage{tikz}
\usetikzlibrary{matrix,arrows,positioning,decorations.markings,decorations.pathmorphing}
\usepackage[colorlinks=true, linktocpage=true, citecolor=red,linkcolor=blue]{hyperref}
\usepackage[shortlabels]{enumitem}

\makeatletter
\newcommand{\hlabel}[1]{%
   \protected@write \@auxout {}%
         {\string \newlabel {#1}{{\theteorema}{\thepage}{}{#1}{}} }%
   \hypertarget{#1}{}
}
\makeatother

\providecommand{\U}[1]{\protect\rule{.1in}{.1in}}

\usepackage{authblk}
\title{Poisson--Poincar\'e reduction for Field Theories}
\author[1,2]{Miguel \'A. Berbel}
\author[1,3]{Marco Castrill\'on L\'opez}
\affil[1]{Dept. \'Algebra, Geometr\'ia y Topolog\'ia, Facultad de CC. Matem\'aticas, Universidad Complutense de Madrid,
Madrid 28040, Spain}
\affil[2]{e-mail: mberbel@ucm.es} \affil[3]{e-mail: mcastri@mat.ucm.es}
\date{\today}                     %%don't need date to appear
\newtheorem{theorem}{Theorem}[section]

\newtheorem{definition}[theorem]{Definition}

\newtheorem{lemma}[theorem]{Lemma}

\newtheorem{proposition}[theorem]{Proposition}
\newtheorem{remark}[theorem]{Remark}

\newcommand{\R}{\mathbb{R}}
\newcommand{\F}{\mathbb{F}}
\newcommand{\sS}{\mathbb{S}}
\newcommand{\g}{\tilde{\mathfrak{g}}}
\newcommand{\A}{\mathcal{A}}
\newcommand{\B}{\mathcal{B}}
\newcommand{\id}{\mathrm{id}}
\newcommand{\Ad}{\mathrm{Ad}}
\newcommand{\ad}{\mathrm{ad}}
\newcommand{\vol}{\mathrm{v}}
\newcommand{\tp}{\tilde{p}}
\newcommand{\hlift}{\mathrm{hor}}
\newcommand{\tsigma}{\tilde{\sigma}}
\newcommand{\ubar}[1]{\underaccent{\bar}{#1}}
\newcommand\restr[2]{{% we make the whole thing an ordinary symbol
  \left.\kern-\nulldelimiterspace % automatically resize the bar with \right
  #1 % the function
  \vphantom{\big|} % pretend it's a little taller at normal size
  \right|_{#2} % this is the delimiter
  }}
\begin{document}
\maketitle

\begin{abstract}
Given a Hamiltonian system on a fiber bundle, there is a Poisson covariant formulation of the Hamilton equations. When a Lie group $G$ acts freely, properly, preserving the fibers of the bundle and the Hamiltonian density is $G$-invariant, we study the reduction of this formulation to obtain an analogue of Poisson--Poincar\'e reduction for field theories. This procedure is related to the Lagrange--Poincar\'e reduction for field theories via a Legendre transformation. Finally, an application to a model of a charged strand evolving in an electric field is given.
\end{abstract}

\textbf{70S05}; 70S10, 58D19, 70G45.

\section{Introduction}

Reduction using symmetry has proven to be an advantageous method to understand a myriad of mechanical systems including the notorious rigid body and the movement of particles in a Yang--Mills field \cite{historySym}. In the Hamiltonian picture of Geometric Mechanics, we learnt from the seminal work of Marsden and Weinstein \cite{MWReduction} that the symplectic structure of an autonomous symmetric Hamiltonian system cannot be directly reduced. The reduction is performed in each regular level set of the corresponding momentum map. The same applies to subsequent generalizations to the non-autonomous case, singular problems and reduction by stages (see \cite{RedstagesPoisson,HamiltonianReduction} and references therein). Hence, if one wants to tackle this reduction from a global point of view in the whole manifold, one focuses on the Poisson bracket, since its reduction does not need to introduce momentum maps \cite{PoissonReduction}.

In contrast to symplectic and Poisson reduction, Lagrangian reduction drops variational principles to quotient spaces rather than geometric structures. Euler--Poincar\'e reduction, a reduction where the configuration space is the Lie group of symmetries, is the best known of these procedures (see \cite{MRbook}). This reduction has been largely extended; for instance, Lagrange--Poincar\'e reduction in \cite{BC, CMR} allows for a configuration space distinct to a Lie group. Furthermore, these variational procedures of reduction in Mechanics have been reasonably adapted to Field Theory with the language of fiber bundles and jet spaces. Covariant Lagrangian reduction for principal bundles was studied in \cite{EPFT2001}, and later on, in \cite{EPsubgroup, cov.red} allowing the group of symmetries to be a subgroup of the structure group of the principal bundle. Afterwards, covariant Lagrangian reduction for arbitrary bundles was attained in \cite{EllisLP} establishing Lagrange--Poincar\'e reduction in Field Theory. See also \cite{FTLP} for the definition of a convenient setting adapted to reduction by stages in this field theoretical formulation.

However, despite the efforts made in this direction (for instance, \cite{MultiGroupoids,RemarksMultiSymp, Gunther, Madsen13}), a complete generalization of the Marsden--Weinstein reduction theorem to the diverse geometrical formulations of Hamiltonian classical Field Theory is yet to be accomplished. The interest of this problem exceeds Field Theory and attracts attention in other topics which make use of general polysymplectic and multisymplectic manifolds. For instance, Lie Systems with compatible geometric structures \cite{LieSystemsMulti}. Recent remarkable  advances can be found in \cite{RouthPolySym, RedPolySymp} for the case of Hamiltonian polysymplectic system, and in \cite{RedMultisymp} for the multisymplectic counterpart. Still, this approach requires the introduction of an arbitrary multimomentum map as those studied in \cite{gotay, Madsen12}. Another promising line of research related to multisymplectic spaces is the reinterpretation of the correspondence principle via a Poisson bracket structure on the space of solutions of the equations of motion of first order Hamiltonian field theories as studied in \cite{SolsHamFT1,SolsHamFT2,souriau}.

The purpose of this paper is to propose a reduction scheme for the Poisson covariant formulation of field theoretic Hamiltonian systems developed in \cite{ibort, PoissonForms, CovariantPoisson}. This reduction aims to generalize the Poisson reduction present in Mechanics for these Hamiltonian field theories. As such, no multimomentum map is introduced and the focus is made on reducing the covariant bracket. The particular case of a system in a principal bundle whose structure group is the symmetry group of the system was already worked out in \cite{someremarks}. In this sense, the reduction procedure presented in \cite{someremarks} can be assimilated to Lie-Poisson reduction while the results of this paper generalize Poisson--Poincar\'e reduction to Field Theory. The main results thus provide reduced brackets and equations that are equivalent to the Poisson bracket and Hamilton equations respectively, under the regular action of a Lie group of symmetries on the dual jet and the polysymplectic bundles. It is important to understand the role of connections in this setting. Recall that (see, for example, \cite{ibort, gotay, CovariantPoisson}) that the Hamiltonian density of a Hamiltonian system requires the introduction of a connection. Moreover, reduction procedures themselves usually require a principal connection for the splitting of objects along the orbits of the action of the symmetry group. They both play auxiliary but essential roles. In this article, we will make use of different connections, but the reader has to keep in mind that they depend, in fact, on these connections associated to the Hamiltonian and the reduction.

This paper is structured as follows:
\begin{itemize}
\item Section \ref{preliminaries} overviews multisymplectic and polysymplectic Field Theory as well as the covariant bracket formalism used in \cite{someremarks, CovariantPoisson}. Observe that we present multisymplectic and polysymplectic bundles appearing in Field Theory rather than general polysymplectic and multisymplectic manifolds.
\item Section \ref{SEC: symmetries} studies the reduced space obtained from a symmetric polysymplectic bundle, finds local expression of the projection in adapted coordinates and determines how Poisson $(n-1)$-forms reduce.
\item Section \ref{SEC: Connections} introduces Ehresmann and linear connections that provide a better geometric understanding of results in Sections \ref{SEC:hor bracket} and \ref{SEC:reduction}.
\item Section \ref{SEC:hor bracket} provides a covariant bracket on the reduced polysymplectic space with different contributions from the covariant bracket on polysymplectic spaces and the Lie--Poisson bracket on Lie algebras.
\item Section \ref{SEC:reduction} introduces Poisson--Poincar\'e reduction after proving that the projection of Section \ref{SEC: symmetries} respects the covariant bracket formulation in the polysymplectic bundle and the bracket introduced in Section \ref{SEC:hor bracket}.
\item Finally, Section \ref{SEC: Legendre} links this reduction with the Lagrange--Poincar\'e reduction proposed in \cite{EllisLP} and Section \ref{SEC: example} illustrates these results with a Hamiltonian model for the time evolution of a charged strand moving in an constant electric field.
\end{itemize}
Throughout this paper, the Einstein summation convention is assumed. For technical simplicity, the base manifold $M$ is assumed to be compact. Given a bundle $E\to M$, the projection map is denoted $\pi_{M,E}$ and the space of sections is denoted $\Gamma(E\to M)$ or simply $\Gamma(E)$ if the base space is clear.

\section{Preliminaries} \label{preliminaries}

\subsection*{Fiber Bundles} Given a fiber bundle $\pi_{M,P}:P\rightarrow M$, we define the following equivalence relation: Two local sections $s: U\to P$, $s': U'\to P$ represent the same jet $j^1_xs$ at $x\in U\cap U'\subset M$ if and only if  $s(x)=s'(x)$ and $T_xs=T_xs'$. We denote $J^1_xP$ the space of classes of equivalence, and the total space $J^1P=\bigcup_{x\in M}J^1_xP$ is called the \em first jet bundle \em of $P\rightarrow M$. This bundle plays the same role in Field Theory as the tangent bundle does in Mechanics. The projection $J^1P\to P,\, j^1_xs\mapsto s(x)$ is an affine bundle modeled over the vector bundle $\pi^*_{M,P}T^*M\otimes VP$, where $VP=\mathrm{ker} \left(T\pi_{M,P}\right)\subset TP$ is the \em vertical bundle. \em For a fiber chart $(x^i,y^a)$ on $P$, $(x^i, y^a, y_i^a)$ denotes fiber coordinates on $J^1P$.

An \em Ehresmann connection \em on $P\rightarrow M$ is a distribution $\Lambda\subset TP$ such that for every  $y\in P$, $T_yP=V_yP\oplus\Lambda_y$. That is, the distribution $\Lambda$ determines \em horizontal vectors \em that complement the vertical subspace $VP$. An Ehresmann connection on $P$ can be defined via a section $y\mapsto j^1_xs$ on $J^1P\to P$ setting $\Lambda_y=\mathrm{Im} \left(T_xs\right)$. For any $x\in M$ and $y\in \pi_{M,P}^{-1}(x)$, $\Lambda_y$ is isomorphic to $T_xM$, through $T_y\pi_{M,P}$. The inverse is called \emph{horizontal lift} and is denoted by $\hlift^{y}_{\Lambda}$.

A $G$\em-principal bundle \em is a fiber bundle obtained from the free and proper left action $\Phi:G\times P\to P$ of a Lie group $G$ on a manifold $P$. In that instance, the quotient $\Sigma=P/G$ is also a manifold and we define $\pi_{\Sigma,P} :P\rightarrow \Sigma=P/G$. A $G$-invariant Ehresmann connection, in the sense that the distribution is invariant under $T\Phi_g$ for any $g\in G$, is called a \textit{principal connection}. The horizontal distribution of a principal connection is the kernel of a $1$-form $\mathcal{A}$ on $P$ taking values on $\mathfrak{g}$, the Lie algebra of $G$, such that:
\begin{enumerate}[(i)]
\item $\mathcal{A}(\xi _{y}^{P})=\xi $, for any $\xi \in\mathfrak{g}$, $y\in P$,
\item  $\Phi_{g}^*\mathcal{A}=\mathrm{Ad}_{g}\circ \mathcal{A}$,
\end{enumerate}
where $\mathrm{Ad}$ denotes the adjoint action of $G$ on $\mathfrak{g}$, and $$\xi ^{P}_{y}=\left. \frac{d}{dt}\right\vert _{t=0}\exp (t\xi )\cdot y\in T_yP,$$
is the infinitesimal generator of $\xi\in\mathfrak{g}$ at $y\in P$. The \textit{curvature} of a connection $\mathcal{A}$ is the $\mathfrak{g}$-valued $2$-form $$B(v,w)=d\mathcal{A}\left(\mathrm{Hor}(v),\mathrm{Hor}(w)\right),$$ where $v,w\in T_yP$ and $\mathrm{Hor}(v)$ is the projection of $v\in T_yP$ to the horizontal subbundle $H_{y}P$.

\subsection*{Multisymplectic Formalism}
The \em dual jet bundle \em $J^1P^* $ of a fiber bundle $\pi_{M,P}:P\rightarrow M$ is a vector bundle over $P$ such that for any $y \in P_x,x\in M$, the fiber is $J^1_yP^*=\mathrm{Aff}(J^1_yP,\bigwedge^nT_x^*M)$, the set of affine maps from the jet bundle to the bundle of $n$-forms on $M$, where $n=\mathrm{dim} M$. Provided a fiber chart $(x^i,y^a)$ on $P$, we define fiber coordinates $(x^i, y^a, p^i_a, \tp)$ on $J^1P^*$ so that its elements have the expression:
\begin{equation*}
y_i^a\mapsto (\tp+p^i_ay^a_i)d^nx,
\end{equation*}
where $d^nx=dx^1\wedge\cdots\wedge dx^n$. Observe that the use of affine morphisms implies that $\mathrm{rank}(J^1P^*)=\mathrm{rank}(J^1P)+1$.

An alternative description of $J^1P^*$ is provided by the subbundle $Z$ of $\bigwedge^nT^*P$ of $n$-forms annihilated when contracted with two vertical vectors. Any $n$-form $z\in Z$ can locally be written as $z=\tp d^nx+p^i_ady^a_i\wedge d^{n-1}x_i$, where $d^{n-1}x_i=i_{\partial/\partial x^i}d^nx$. The mapping $\Psi:Z\to J^1 P^*$ defined by $$\langle\Psi(z),j^1_xs\rangle=s^*z\in \bigwedge^nT^*_xM,$$ is a vector bundle isomorphism whose localization in the coordinate systems previously introduced in $Z$ and $J^1P^*$ is the identity. In conclusion, both spaces are canonically isomorphic as vector bundles over $P$.

There is a \em canonical \em $n$-form $\Theta_{\wedge}$ on $\bigwedge^nT^*P$ defined as
\begin{equation*}
\Theta_{\wedge}(z)(u_1,\dots,u_n)=z\left(T\pi_{P,\bigwedge^nT_x^*P}u_1,\dots,T\pi_{P,\bigwedge^nT_x^*P}u_n\right), \quad z\in Z,
\end{equation*}
where $u_i\in T_z(\bigwedge^nT^*P)$ for $i=1,\dots,n$. If $i_{Z,\bigwedge^nT^*P}:Z \to\bigwedge^nT^*P$ denotes the inclusion, the canonical $n$-form $\Theta$ on $Z$ is the pullback $\Theta=i^*_{Z,\bigwedge^nT^*P}\Theta_{\wedge}$ and the \em canonical multisymplectic \em $(n+1)$-form
\begin{equation}
\Omega=-d\Theta.
\end{equation}
The canonical $n$-form $(\Psi^{-1})^*\Theta$ and  $(n+1)$-form $(\Psi^{-1})^*\Omega$ on $J^1P^*$ will again be denoted $\Theta$ and $\Omega$ respectively. The identification of $J^1P^*$ and $Z$ via $\Psi$ justifies this abuse of notation. These forms have the following local expressions:
\begin{equation}
\Theta=p_a^idy^a\wedge d^{n-1}x_i+\tp d^nx,
\end{equation}
\begin{equation}
\Omega=dy^a\wedge dp_a^i \wedge d^{n-1}x_i-d\tp\wedge d^nx.
\end{equation}

The pair $(J^1P^*,\Omega)$ is an example of \em multipsymplectic manifold \em called canonical $n$-plectic structure on $J^1P^*$. A more detailed discussion of the multisyplectic formalism can be found on \cite{gotay} and a theory on the reduction of general multisymplectic manifolds is explained in \cite{RedMultisymp}.

\subsection*{Polysymplectic Formalism}

Given a fiber bundle $P\to M$, the \em polysymplectic bundle \em is the following vector bundle over $M$
\begin{equation*}
\Pi_P=\pi^*_{MP}TM\otimes_PV^*P\otimes_P\pi^*_{MP}\left(\bigwedge^nT^*M\right).
\end{equation*}
Unless the base space is not clear from the context, we shall simply write $\Pi_P=TM\otimes V^*P\otimes\bigwedge^nT^*M$ to avoid tedious notation. For every $p\in \Pi_{P}$, there are local coordinates $(x^i, y^a, p^i_a)$ such that
\begin{equation*}
p=p^i_a\frac{\partial}{\partial x^i}\otimes dy^a \otimes d^nx.
\end{equation*}
Observe that $\mathrm{rank}(\Pi_P)=\mathrm{rank}(J^1P)$ as bundles over $P$. In addition, the fiber map
\begin{align*}
J^1P^*&\to \Pi_P\\
\varphi&\mapsto \vec{\varphi},
\end{align*}
where $\vec{\varphi}\in (\Pi_P)_y=(T_x^*M\otimes V_yP)^*\otimes\bigwedge^nT_x^*M$ is the linear map asociated to the affine map $\varphi\in J^1_yP^*=\mathrm{Aff}(J^1_yP,\bigwedge^nT_x^*M)$, is an affine bundle modeled on the rank-one vector bundle $\bigwedge^nT^*M\to M$.

A \em Hamiltonian system \em is a pair $(\Pi_P,\delta)$ where $\delta$ is a section of $J^1P^*\to \Pi_P$. The canonical forms on $J^1P^*$ can be pulled back to $\Pi_P$ to obtain $\Theta_{\delta}=\delta^*\Theta$ and $\Omega_\delta=-d\Theta_{\delta}$. In local coordinates, $\delta(x^i,y^a, p^i_a)=(x^i, y^a,p^i_a, H_{\delta}(x^i,y^a, p^i_a))$ and
\begin{equation}\label{pullback canform}
\Theta_{\delta}=p_a^idy^a\wedge d^{n-1}x_i+H_{\delta}(x^i,y^a, p^i_a)d^nx.
\end{equation}
A section $p$ of $\Pi_P\to M$ is a \em solution \em of the Hamiltonian system if for any vertical vector field $X$ on $\Pi_P$,
\begin{equation} \label{Hameq}
0=p^*i_X\Omega_\delta=-p^*i_Xd\Theta_\delta.
\end{equation}

Given an Ehresmann connection $\Lambda$ in $P\to M$, there is a natural section $\delta_{\Lambda}$ of $J^1P^*\to\Pi_P$ such that for any $v_x\otimes\omega_y \otimes\vol\in T_xM\otimes V_y^*P\otimes\bigwedge^nT_x^*M$,
\begin{equation*}
\delta_{\Lambda}(v_x\otimes\omega_y \otimes\vol)=\langle \omega_y,\mathrm{Hor}(\cdot)\rangle \wedge i_{v_x}\vol\in Z_y\simeq (J^1P^*)_y.
\end{equation*}
For any section $\delta$ of $J^1P^*\to \Pi_P$, there is a mapping $\mathcal{H}:\Pi_P\to\bigwedge^nT^*M$ called  \em Hamiltonian density \em such that $\mathcal{H}=\delta-\delta_{\Lambda}$. Thus, a Hamiltonian system can equivalently be defined by a triple $(\Pi_P,\Lambda,\mathcal{H})$ where $\Lambda$ is an Ehresmann connection and $\mathcal{H}$ a Hamiltonian density. The Hamilton equation \eqref{Hameq} is locally equivalent to the Hamilton--Cartan equations
\begin{equation} \label{Hamilton-Cartan}
\frac{\partial H}{\partial p^i_a}=\frac{\partial y^a}{\partial x^i}-\Lambda^a_i; \hspace{5mm}
-\frac{\partial H }{\partial y^a}=\frac{\partial p^i_a}{\partial x^i}+\frac{\partial \Lambda^b_i}{\partial y^a}p^i_b,
\end{equation}
where $\Lambda^a_i$ are the coefficients of the horizontal lift defined by $\Lambda$ and $\mathcal{H}=Hd^nx$.

An $r$-form $F$ on $J^1P^*$ is said to be \em horizontal \em if $i_uF=0$ for any vertical tangent vector with respect to $J^1P^*\to M.$ In local coordinates,
$$F=F_{i_1,\dots, i_r}dx^{i_1}\wedge\dots\wedge dx^{i_r}.$$
If there is a vertical $(n-r)$-multivector field $\chi_F$ on $J^1P^*$ such that
\begin{equation}
i_{\chi_F}\Omega =dF,
\end{equation}
$F$ is said to be a \em Poisson form \em \cite{PoissonForms}. Given $F$ a Poisson $r$-form and $E$ a Poisson $s$-form their \em Poisson bracket \em is the $(r+s+1-n)$-form on $J^1P^*$,
\begin{equation}\label{Poisson bracket}
\{F,E\}=(-1)^{r(s-1)}i_{\chi_E}i_{\chi_F}\Omega.
\end{equation}
which, in turn, is a Poisson form. To be exact, this operator is a generalized graded Poisson bracket with a modified Leibniz rule giving rise to a higher-order Gerstenhaber algebra. More details can be found in \cite{PoissonForms}. It has been proved in \cite{someremarks} that a Poisson form on $J^1P^*$ projects to $\Pi_P$ so that we can consider that Poisson forms are defined on $\Pi_P$. Poisson $(n-1)$-forms are key to express the Hamilton-Cartan equations \eqref{Hamilton-Cartan} in terms of a bracket in the same way affine functions in a cotangent bundle play a fundamental role in Hamiltonian Mechanics:
\begin{theorem}\emph{\cite[Proposition 5.2]{someremarks}} \label{TH bracket formulation}
A section $p$ of $\Pi_P\to M$ is a solution of a given Hamiltonian system $(\Pi_P,\Lambda,\mathcal{H})$, $\mathcal{H}=H\vol$ if and only if for any horizontal Poisson $(n-1)$-form $F$ the following equation holds true:
\begin{equation} \label{Hameqs with bracket}
\{F,H\}\vol\circ p=d(p^*F)-(d^hF)\circ p,
\end{equation}
where $d^hF$ is the horizontal differential of $F$ with respect to the connection on $\Pi_P$.
\end{theorem}
The aforementioned connection on $\Pi_P\to M$ is obtained from the Ehresmann connection $\Lambda$ on $P\to M$ and any Riemannian connection on $M$. The local expression of the horizontal lift of this connection is
\begin{align} \label{conn.poly}
\frac{\partial}{\partial x^i}\mapsto\frac{\partial}{\partial x^i}+\Lambda^a_i\frac{\partial}{\partial y^a}+\left(-\frac{\partial \Lambda^b_i}{\partial y^a}p^j_b+\Gamma^j_{ik}p^k_a-\Gamma^k_{ik}p^j_a\right)\frac{\partial}{\partial p^j_a},
\end{align}
where $\Gamma^k_{ij}$ are the Christoffel symbols of the Riemannian connection. As stated before, the main objective of this article is to present a reduction theory for this Poisson covariant formulation of the Hamilton equations.

\section{The Reduced Polysymplectic Bundle} \label{SEC: symmetries}

Given a fiber bundle $P\to M$ and a Lie group $G$ acting freely, properly and vertically on $P$, we have $P\to P/G=\Sigma\to M$, where $P\to\Sigma$ is a principal bundle and $\Sigma\to M$ a fiber bundle. Let $\A$ be a principal connection on $P\to\Sigma$, the map
\begin{align*}
\alpha_{\A}: (TP)/G\to &T\Sigma\oplus \g \\
[\dot{y}]_G \mapsto &T\pi_{\Sigma,P}(\dot{y})\oplus [y, \A(\dot{y})]_G,
\end{align*}
where $\g= (P\times\mathfrak{g})/G$ is the adjoint bundle obtained as an associated bundle from the adjoint action of $G$ on $\mathfrak{g}$, is an isomorphism of vector bundles over $\Sigma$ whose inverse is
\begin{align*}
\alpha^{-1}_{\A}:T\Sigma\oplus\g\to & (TP)/G \\
\dot{s}\oplus [y,\xi]_G \mapsto& [\hlift_{\A}(\dot{s})+ \xi^P_y]_G.
\end{align*}
Here $\hlift_{\A}:\pi^*_{\Sigma,P}T\Sigma\to TP$ denotes the horizontal lift induced by the connection $\A$. Observe that $\alpha_{\A}$ is a diffeomorphism of bundles over $M$. The dual morphism of $\alpha^{-1}_{\A}$, $\tau_{\A}=(\alpha^{-1}_{\A})^*$, similarly splits $(T^*P)/G$ in a vertical and horizontal part with respect to $\A$. It follows that
$$\tau_{\A}=[\hlift^*_{\A}]_G\oplus[\ubar{J}]_G:(T^*P)/G \to  T^*\Sigma\oplus\g^*$$
where $\hlift^*_{\A}:T^*P\to \pi^*_{\Sigma,P}T^*\Sigma$ is the dual morphism of the horizontal lift $\hlift_{\A}$ and $[\ubar{J}]_G:(T^*P)/G\to \g^*$ is the function induced by the equivariant momentum map of the action on $T^*P$
\begin{align*}
\ubar{J}:T^*P\to &\mathfrak{g}^* \\
\gamma_y\mapsto &\ubar{J}(\gamma_y)(\xi)=\gamma_y(\xi^P_y).
\end{align*}
The map $\tau_{\A}$ is an isomorphism of vector bundles over $\Sigma$ whose inverse is:
\begin{align*}
\tau^{-1}_{\A}:T^*\Sigma\oplus\g^*\to & (T^*P)/G \\
\beta_s\oplus\eta_s \mapsto& \left[\pi^*_{\Sigma,P}\beta_s+\eta_s\circ [p,\A(\cdot)]\right]_G.
\end{align*}

There are two different notions of vertical subspace of $TP$ depending on which base space is considered. On one hand, $V_{\pi_{M,P}}P=\{v\in TP,\, T\pi_{M,P}(v)=0\}$ are the vertical vectors with respect to the fiber bundle $\pi_{M,P}:P\to M$. On the other hand, $V_{\pi_{\Sigma,P}}P=\{v\in TP,\, T\pi_{\Sigma,P}(v)=0\}$, where vertical refers to the bundle $\pi_{\Sigma,P}:P\to \Sigma$. Hereafter, the notation $VP$ will exclusively refer to $V_{\pi_{M,P}}P$.

\begin{proposition} \label{restr} The restricted map $\alpha_{\A}\vert_{VP/G}: (VP)/G\to V\Sigma\oplus \g $ is an isomorphism of vector bundles over  $\Sigma$ and a diffeomorphism of bundles over $M$, whose inverse is given by $\alpha_{\A}^{-1}\vert_{V\Sigma\oplus\g}$.
\end{proposition}
\begin{proof}
For every $v_y\in T_yP$, $$T_y\pi_{M,P}(v_y)=T_s\pi_{M,\Sigma}\circ T_y\pi_{\Sigma,P}(v_y),$$ and $v_y\in V_yP$ if and only if $T\pi_{\Sigma,P}(v_y)\in V_s\Sigma$. Consequently, $\mathrm{Im}\left(\alpha_{\A}\vert_{VP/G}\right)\subseteq V\Sigma\oplus \g $.
For every $v_s\in T_s\Sigma$,
$$T_s\pi_{M,\Sigma}(v_s)=T_s\pi_{M,\Sigma}\circ T_y\pi_{\Sigma,P}\circ \hlift^y_{\A}(v_s)=T_y\pi_{M,P}\circ \hlift^y_{\A}(v_s).$$
Therefore, $v_s\in V_s\Sigma$ if and only if $\hlift^y_{\A}(v_s)\in V_yP$, $\restr{\hlift_{\A}}{V\Sigma}:V\Sigma\to VP$ and $\mathrm{Im}(\alpha_{\A}^{-1}\vert_{V\Sigma\oplus \g })\subseteq VP/G$.
As $\alpha_{\A}^{-1}\circ \alpha_{\A}= \mathrm{Id}_{TP/G}$, we observe that
$$\mathrm{Id}_{VP/G}= (\alpha_{\A}^{-1}\circ \alpha_{\A})\vert_{VP/G}= \alpha_{\A}^{-1}\vert_{V\Sigma\oplus\g}\circ \alpha_{\A}\vert_{VP/G}$$
and conclude that both maps are in fact isomorphisms.
\end{proof}

We shall denote  $\alpha_{\A}: (VP)/G\to V\Sigma\oplus \g $ and $\alpha_{\A}^{-1}:V\Sigma\oplus \g \to (VP)/G$ the restricted maps introduced in Proposition \ref{restr}. Observe that this is an abuse of notation since we already established that notation for the unrestricted maps. Correspondingly, we shall denote
$$\tau^{-1}_{\A}:V^*\Sigma\oplus\g^*\to  (V^*P)/G, \qquad \tau_{\A}:(V^*P)/G \to  V^*\Sigma\oplus\g^*,$$
their respective dual maps. These splittings induce a similar one in the polysymplectic bundle.
\begin{proposition} \label{quotient.ident}
Given a principal connection $\A$ on the fiber bundle $P\rightarrow\Sigma$, the map
\begin{align*}
\psi_{\A}=\id_{TM}\otimes\tau_{\A}\otimes\id_{\bigwedge^{n}T^*M}:\Pi_{P}/G  & \longrightarrow\Pi_{\Sigma}\oplus (TM\otimes\mathfrak{\tilde{g}%
}^*\otimes\Lambda^{n}T^*M)
\end{align*}
is an isomorphism of vector bundles over $\Sigma$.
\end{proposition}
%\begin{proof}
%Since $G$ acts on $P\to M$ via bundle isomorphisms that project to the identity on $M$, $G$ acts trivially on both, $TM$ and $\bigwedge^{n}T^*M$. Thus, the quotient of the polysymplectic space $\Pi_P=TM\otimes_M V^*P\otimes_M \bigwedge^{n}T^*M$  is $$\Pi_P/G=TM\otimes_{\Sigma} (V^*P)/G\otimes_{\Sigma} \bigwedge^{n}T^*M,$$
%and $\psi_{\A}=\id_{TM}\otimes\tau_{\A}\otimes\id_{\bigwedge^{n}T^*M}$ is an isomorphism of vector bundles over $\Sigma$.
%\end{proof}

We shall define $\kappa= \psi_{\A}\circ\pi_{\Pi_P/G,\Pi_P}$, the projection of $\Pi_P\to \Pi_P/G$ using the identification $\psi_{\A}$. In the following, we obtain a local expression for $\kappa$ and $T\kappa$ in an adapted coordinate system.

Suppose that the principal bundle $P\to \Sigma$ is trivializable and $P\simeq \Sigma\times G$. Let $\Sigma\to M$ be in turn trivializable and $\Sigma\simeq M\times F$, where $F$ is diffeomorphic to the fibers. For any point $y\in P$, we will respectively denote its projection to $\Sigma$ and $M$ by $s=\pi_{P,\Sigma}(y)$  and $x=\pi_{M,P}(y)$. A trivialization such that $y=(s,e)$ can always be chosen. In coordinates we can express $s=(x^i,x^a)$ where $i=1,\dots,n=\dim M $, $k=\dim G$, $m=\dim P > n+k$,  and $a=1,\dots,m-k-n$. Also, $x=\pi_{M,P}(y)=\pi_{\Sigma,P}(s)$ has coordinates $(x^1, \dots, x^i, \dots, x^n)$. Fixing a basis $\{B_1,\dots,B_k\}$ of $\mathfrak{g}$ we can choose a normal coordinate system in a neighborhood of $(s,e)$ such that $\mathrm{v}=dx^1\wedge\dots\wedge dx^n$ is a volume form and the coordinates of any point $(s,g)$ in such neighborhood are $(x^i,x^a,y^{\alpha})$ where $y^\alpha(g)$, for $\alpha=1,\dots,k$ are given by
$$g=\exp(y^{\alpha}(g)B_{\alpha}).$$
The induced coordinate system on $\Pi_P$ has coordinates $(x^i,x^a,y^{\alpha},p^i_a,p^i_{\alpha})$. Similarly, since the coordinates in $\Sigma$ are given by $(x^i,x^a)$, the induced coordinate system on $\Pi_{\Sigma}$ is $(x^i,x^a,\sigma^i_a)$. The trivialization around $y=(s,e)$ defines local sections $\tilde{B}_{\alpha}$ in $\g\to\Sigma$ given by $s\mapsto [(s,e),B_{\alpha}]_G$ as well as local sections  $\tilde{B}^{\alpha}$ in $\g^*\to\Sigma$ and $\tilde{B}^{\alpha}_i=\frac{\partial}{\partial x^i}\otimes\tilde{B}^{\alpha}\otimes\mathrm{v}$ in $TM\otimes\mathfrak{\tilde{g}%
}^*\otimes\bigwedge^{n}T^*M$. The later set of local sections induces a local coordinate system ($x^i,x^a,\mu^i_{\alpha})$ on $TM\otimes\mathfrak{\tilde{g}}^*\otimes\bigwedge^{n}T^*M$. We will not restrict the results on this paper to trivializable bundles, yet, since bundles are locally trivializable, we can define such sets of coordinates around any $y\in P$ and use them to prove local formulas.

In particular, the horizontal lift at $y$ through $\A$, $\hlift^y_{\A}$, has a local expression given by
\begin{align*}
\hlift_{\A}^y:\pi_{\Sigma,P}^*T_s\Sigma\to & T_yP\\
\frac{\partial}{\partial x^i} \mapsto &\frac{\partial}{\partial x^i}+A^{\alpha}_i(y)\frac{\partial}{\partial y^{\alpha}}\\
\frac{\partial}{\partial x^a} \mapsto &\frac{\partial}{\partial x^a}+A^{\alpha}_a(y)\frac{\partial}{\partial y^{\alpha}},
\end{align*}
and we have the following result.
\begin{proposition} \label{localkappa}
The local expression of the projection $\kappa:\Pi_P \to\Pi_{\Sigma}\oplus (TM\otimes\mathfrak{\tilde{g}%
}^*\otimes\bigwedge^{n}T^*M)$ is:
\begin{equation*}
\kappa(x^i,x^a,y^{\alpha},p^i_a,p^i_{\alpha})=(x^i,x^a,\sigma^i_a=p^i_a+p^i_{\beta}A^{\beta}_a(x^i,x^a,y^{\lambda}),\mu^i_{\alpha}=Z^{\beta}_{\alpha}(g^{-1})p^i_{\beta}),
\end{equation*}
where $Z^{\beta}_{\alpha}(g)$ is such that
$$T_gL_{g^{-1}}\left(\frac{\partial}{\partial y^{\alpha}}\right)_g=Z^{\beta}_{\alpha}(g)\left(\frac{\partial}{\partial y^{\beta}}\right)_e. $$
In particular, for $y=(s,e)$; $\sigma^i_a=p^i_a+p^i_{\alpha}A^{\alpha}_a$ and $\mu^i_{\alpha}=p^i_{\alpha}$.
\end{proposition}
\begin{proof}
First we obtain the local expression of $\tau^{-1}_{\A}:T^*\Sigma\oplus\g^*\to (T^*P)/G$. From its definition
$$\tau^{-1}_{\A}(dx^a)=\pi^*_{\Sigma,P}dx^a=dx^a.$$
Let $y=(s,g)$ for every $v\in T_yP$,
\begin{align*}
\tau^{-1}_{\A}(\tilde{B}^{\alpha})([v]_G)&=\langle\tilde{B}^{\alpha},[(s,g),\mathcal{A}(s,g)(v)]_G\rangle =\langle\tilde{B}^{\alpha},[(s,e),\Ad_{g^{-1}}\circ\mathcal{A}(s,e)(v)]_G\rangle\\&=\langle\tilde{B}^{\alpha},[(s,e),\mathcal{A}(s,e)\circ T_gL_{g^{-1}}(v)]_G\rangle.
\end{align*}
Since
\begin{align*}
\mathcal{A}(s,e)\circ T_gL_{g^{-1}}\left(\frac{\partial}{\partial x^a}\right)_g&=\mathcal{A}(s,e)\left(\frac{\partial}{\partial x^a}\right)_e=-A^{\beta}_a(s,e),\\
\mathcal{A}(s,e)\circ T_gL_{g^{-1}}\left(\frac{\partial}{\partial y^{\beta}}\right)_g&=\mathcal{A}(s,e)\left( Z^{\alpha}_{\beta}(g)\left(\frac{\partial}{\partial y^{\alpha}}\right)_e\right)=\delta^{\gamma}_{\beta}Z^{\alpha}_{\gamma}(g),
\end{align*}
we obtain that
$$\tau^{-1}_{\A}(\tilde{B}^{\alpha})=-A^{\alpha}_a(s,e)dx^a+Z^{\alpha}_{\beta}(g)dy^{\beta}$$
and we complete the local expression of $\tau^{-1}_{\A}$. Then, we get the local expression of $\tau_{\A}$
\begin{align*}
\tau_{\A}:(T^*P)/G \to&  T^*\Sigma\oplus\g^*\\
dx^a\mapsto& dx^a\\
dy^{\beta}\mapsto& Z^{\beta}_{\alpha}(g^{-1})\tilde{B}^{\alpha}+Z^{\beta}_{\gamma}(g^{-1})A^{\gamma}_a(s,e)dx^a=Z^{\beta}_{\alpha}(g^{-1})\tilde{B}^{\alpha}+A^{\beta}_a(s,g)dx^a
\end{align*}
Consequently,
$$\kappa\left(p^i_{\beta}\frac{\partial}{\partial x^i}\otimes dy^{\alpha}\otimes\vol\right)=
p^i_{\beta}A^{\beta}_a(s,g)\frac{\partial}{\partial x^i}\otimes dx^a\otimes\vol+
p^i_{\beta}Z^{\beta}_{\alpha}(g^{-1})\frac{\partial}{\partial x^i}\otimes \tilde{B}^{\alpha}\otimes\vol,$$
$$\kappa\left(p^i_a\frac{\partial}{\partial x^i}\otimes dx^a\otimes\vol\right)=p^i_a\frac{\partial}{\partial x^i}\otimes dx^a\otimes\vol.$$
and the local expression follows.
\end{proof}

\begin{lemma} \label{baker-campbell}
The transformation $Z^{\beta}_{\alpha}(g)$ can be expressed as
$$Z^{\beta}_{\alpha}(g)=\sum^{\infty}_{n=1}\left.\frac{d}{d\varepsilon}\right\vert_{\varepsilon =0}c^{\beta}_n(-y^{\gamma}B_{\gamma}:y^{\gamma}B_{\gamma}+\varepsilon B_{\alpha}),$$
where $c_n$ are the coefficients of the Baker-Campbell-Hausdorff formula \begin{equation*}
\exp X\exp Y=\exp\left( \sum^{\infty}_{n=1} c_n(X:Y)\right), \quad X,Y\in\mathfrak{g}.
\end{equation*}
Furthermore,
$$\left.\frac{\partial Z^{\beta}_{\alpha}}{\partial y^{\gamma}}\right\vert_{y=(s,e)}=\frac{1}{2}c^{\beta}_{\alpha\gamma},$$
where $c^{\beta}_{\alpha\gamma}$ are the coefficients of the Lie bracket in $\mathfrak{g}$.
\end{lemma}
\begin{proof}
For the first part, observe that
\begin{align*}
T_gL_{g^{-1}}\left(\frac{\partial}{\partial y^{\alpha}}\right)_g=& T_gL_{g^{-1}} \left(\left.\frac{d}{d\varepsilon}\right\vert_{\varepsilon =0}\exp(y^{\gamma}B_{\gamma}+\varepsilon B_{\alpha})\right)\\
=& \left.\frac{d}{d\varepsilon}\right\vert_{\varepsilon =0} \exp(-y^{\gamma}B_{\gamma})\exp(y^{\gamma}B_{\gamma}+\varepsilon B_{\alpha})\\
=& \sum^{\infty}_{n=1}\left.\frac{d}{d\varepsilon}\right\vert_{\varepsilon =0}c^{\beta}_n(-y^{\gamma}B_{\gamma}:y^{\gamma}B_{\gamma}+\varepsilon B_{\alpha})\left(\frac{\partial}{\partial y^{\beta}}\right)_e.
\end{align*}
Then,
\begin{align*}
\left.\frac{\partial Z^{\beta}_{\alpha}}{\partial y^{\gamma}}\right\vert_{y=(s,e)}=&\left.\frac{d}{d \epsilon}\right\vert_{\epsilon =0}Z^{\beta}_{\alpha}(\exp(\epsilon B_{\gamma}))=\left.\frac{d}{d \epsilon}\right\vert_{\epsilon =0}\sum^{\infty}_{n=1}\left.\frac{d}{d\varepsilon}\right\vert_{\varepsilon =0}c^{\beta}_n(-\epsilon B_{\gamma}:\epsilon B_{\gamma}+\varepsilon B_{\alpha})\\
=&\left.\frac{d}{d \epsilon}\right\vert_{\epsilon =0}\left.\frac{d}{d\varepsilon}\right\vert_{\varepsilon =0}c^{\beta}_2(-\epsilon B_{\gamma}:\epsilon B_{\gamma}+\varepsilon B_{\alpha})\\=&\left.\frac{d}{d \epsilon}\right\vert_{\epsilon =0}\left.\frac{d}{d\varepsilon}\right\vert_{\varepsilon =0}\frac{1}{2}[-\epsilon B_{\gamma},\epsilon B_{\gamma}+\varepsilon B_{\alpha}]^{\beta}=-\frac{1}{2}c^{\beta}_{\gamma\alpha}=\frac{1}{2}c^{\beta}_{\alpha\gamma},
\end{align*}
where we have used that $c^{\beta}_n(-\epsilon B_{\gamma}:\epsilon B_{\gamma}+\varepsilon B_{\alpha})$ is an homogeneous polynomial of degree $n$ with respect to $\varepsilon$ and $\epsilon$ in order to obtain the second-to-last equivalence.
\end{proof}

\begin{lemma} \label{func.inv}
Let $E:\Pi_P\to \R$ be a $G$-invariant real function on $\Pi_P$, then
$$0=\frac{\partial E}{\partial y^{\alpha}}+\frac{1}{2}c^{\gamma}_{\alpha\beta}\frac{\partial E}{\partial p^i_\beta}.$$
Similarly, let $C: P\to \mathfrak{g}$ be a $G$-equivariant function on $P$, then
$$0=\frac{\partial C^{\gamma}}{\partial y^{\alpha}}+\frac{1}{2}c^{\gamma}_{\beta\alpha}C^{\beta}.$$
\end{lemma}
\begin{proof}
Since $E$ is invariant under the action of $G$ on $\Pi_P$, for any $g\in G$,
$$E(x^i,x^a,0,p^i_a,p^i_{\beta})=E(g\cdot(x^i,x^a,0,p^i_a,p^i_{\beta}))=E(x^i,x^a, y^{\lambda}, p^i_a,Z^{\gamma}_{\beta}(g)p^i_{\gamma}).$$
In particular, for $g(\varepsilon)=\exp(\varepsilon B_{ \alpha})$, $y^{\alpha}=\varepsilon$, $y^{\beta}=0$ for $\beta\neq\alpha$ and
$$E(x^i,x^a,0,p^i_a,p^i_{\beta})=E(x^i,x^a, y^{\lambda}(\varepsilon), p^i_a,Z^{\gamma}_{\beta}(g(\varepsilon))p^i_{\gamma}).$$
Derivation with respect to $\varepsilon$ and Lemma \ref{baker-campbell} finish the first part of the lemma.
In a similar way, as $C$ is equivariant, for any $g\in G$, $$C^{\beta}(x^i,x^a, y^{\lambda})=Z^{\beta}_{\gamma}(g^{-1})C^{\gamma}(x^i,x^a,0).$$ In particular, it holds for $g(\varepsilon)=\exp(\varepsilon B_{ \alpha})$ and derivation with respect to  $\varepsilon$ concludes the proof.
\end{proof}

\begin{proposition} \label{dif.kappa}
The local expression of the differential of the projection $\kappa:\Pi_P \to\Pi_{\Sigma}\oplus (TM\otimes\g^*\otimes\bigwedge^{n}T^*M)$ at point $y=(s,e)$ is
\begin{align*}
T\kappa: T\Pi_P \to& T(\Pi_{\Sigma}\oplus (TM\otimes\g^*\otimes\bigwedge^{n}T^*M))\\
\frac{\partial}{\partial x^i}\mapsto &\frac{\partial}{\partial x^i} + \mu^j_{\gamma}\frac{\partial A^{\gamma}_b}{\partial x^i}\frac{\partial}{\partial \sigma^j_b}\\
\frac{\partial}{\partial x^a}\mapsto &\frac{\partial}{\partial x^a} + \mu^j_{\gamma}\frac{\partial A^{\gamma}_b}{\partial x^a}\frac{\partial}{\partial \sigma^j_b}\\
\frac{\partial}{\partial y^{\alpha}}\mapsto &-\frac{1}{2}\mu^{j}_{\gamma}c^{\gamma}_{\beta\alpha}A^{\beta}_b\frac{\partial}{\partial \sigma^j_b}-\frac{1}{2}\mu^{j}_{\gamma}c^{\gamma}_{\beta\alpha}\frac{\partial}{\partial \mu^j_{\beta}}\\
\frac{\partial}{\partial p^i_a}\mapsto &\frac{\partial}{\partial \sigma^i_a}\\
\frac{\partial}{\partial p^i_{\alpha}}\mapsto & A^{\alpha}_{b}\frac{\partial}{\partial \sigma^i_b}+\frac{\partial}{\partial \mu^i_{\alpha}}\\
\end{align*}
\end{proposition}
\begin{proof}
%Derivation of the local expression in proposition \ref{localkappa} gives
%\begin{align*}
%T\kappa: T\Pi_P \to& T(\Pi_{\Sigma}\oplus (TM\otimes\g^*\otimes\bigwedge^{n}T^*M))\\
%\frac{\partial}{\partial x^i}\mapsto &\frac{\partial}{\partial x^i} + p^j_{\gamma}\frac{\partial A^{\gamma}_b}{\partial x^i}\frac{\partial}{\partial \sigma^j_b}\\
%\frac{\partial}{\partial x^a}\mapsto &\frac{\partial}{\partial x^a} + p^j_{\gamma}\frac{\partial A^{\gamma}_b}{\partial x^a}\frac{\partial}{\partial \sigma^j_b}\\
%\frac{\partial}{\partial y^{\alpha}}\mapsto &p^{j}_{\gamma}\frac{\partial A^{\gamma}_b}{\partial y^{\alpha}} \frac{\partial}{\partial \sigma^j_b}+p^{j}_{\gamma}\frac{\partial Z^{\gamma}_{\beta}}{\partial y^{\alpha}}\frac{\partial}{\partial \mu^j_{\beta}}\\
%\frac{\partial}{\partial p^i_a}\mapsto &\frac{\partial}{\partial \sigma^i_a}\\
%\frac{\partial}{\partial p^i_{\alpha}}\mapsto & A^{\alpha}_{b}\frac{\partial}{\partial \sigma^i_b}+ Z^{\alpha}_{\beta}(g^{-1})\frac{\partial}{\partial \mu^i_{\beta}},\\
%\end{align*}
%which can be easily rewritten as the searched expression applying lemma \ref{func.inv} to the invariance of $\mathcal{A}$, the second part of lemma \ref{baker-campbell} and evaluation at $y=(s,e)$.
The derivative of the local expression in Proposition \ref{localkappa} can be easily obtained by applying Lemma \ref{func.inv} to the invariance of $\mathcal{A}$, the second part of 	 Lemma \ref{baker-campbell} and evaluation at $y=(s,e)$.
\end{proof}

As stated in Theorem \ref{TH bracket formulation}, horizontal Poisson $(n-1)$-forms on $\Pi_P$ can be used to describe the dynamics of a Hamiltonian system. Let $V$ be any vector bundle over $P$ or $\Sigma$, we introduce a linear operator between vector bundles
\begin{align*}
\llangle\cdot,\cdot\rrangle: \left(TM\otimes V^*\otimes \bigwedge^{n}T^*M\right)\otimes V \to& \bigwedge^{n-1}T^*M \\
\left(\frac{\partial}{\partial x^i}\otimes w\otimes \mathrm{v},v\right)\mapsto& \langle w, v\rangle i_{\partial /\partial x^i}\mathrm{v},
\end{align*}
and for any $\gamma\in\Gamma(V)$, the $(n-1)$-form $\theta_{\gamma}$ is defined on $q\in TM\otimes V^*\otimes \bigwedge^{n}T^*M$ by $(\theta_{\gamma})_q=\llangle q, \gamma \rrangle$. It is known that:
\begin{proposition}\label{n-1.formas unred}\emph{\cite[Proposition 4.3]{someremarks}}
Any Poisson $(n-1)$-form on $\Pi_{P}$ can be written as
$$F=\theta_X+\pi^*_{P,\Pi_P}\omega+\Upsilon,$$
where $X\in\mathfrak{X}^V(P)=\Gamma (VP)$ is a vertical vector field on $P$, $\omega$ is an horizontal $(n-1)$-form on $P$ and $\Upsilon$ is a closed horizontal $(n-1)$-form on $\Pi_P$.
\end{proposition}
We shall see that these objects project to the reduced polysymplectic space.
\begin{theorem} \label{red.n-1.forms}
The projection of a $G$-invariant horizontal Poisson $(n-1)$-form on $\Pi_P$,
$F=\theta_X+\pi^*_{P,\Pi_P}\omega+\Upsilon$
is an horizontal $(n-1)$-form on $\Pi_{\Sigma}\oplus (TM\otimes\g^*\otimes\bigwedge^{n}T^*M)$ of the kind
$$f=\theta_Y+\theta_{\bar{\xi}}+\bar{\pi}^*\bar{\omega}+\bar{\Upsilon},$$
where $Y\oplus\bar{\xi}\in\Gamma(V\Sigma\oplus\g\to M)$ such that $\alpha_{\A}([X]_G)=Y\oplus\bar{\xi}$, $\bar{\omega}$ is an horizontal $(n-1)$-form on $\Sigma$, $\bar{\Upsilon}$ is a closed horizontal $(n-1)$-form on $\Pi_{\Sigma}\oplus (TM\otimes\g^*\otimes\bigwedge^{n}T^*M)$, and $\bar{\pi}=\pi_{\Sigma,\Pi_{\Sigma}\oplus (TM\otimes\g^*\otimes\bigwedge^{n}T^*M)}$.
\end{theorem}
\begin{proof}
Since $F$ is invariant under the action of $G$ on $\Pi_P$, $X$ is invariant under the action of $G$ on $P$, $\omega\in\Omega^{n-1}(P)$ is $G$-invariant and $\Upsilon$ is a $G$-invariant $(n-1)$-form on $\Pi_P$. From the last statement, as $\Upsilon$ is horizontal there is a $(n-1)$-form on $\Pi_{\Sigma}\oplus (TM\otimes\g^*\otimes\bigwedge^{n}T^*M)$, $\bar{\Upsilon}$ such that $\Upsilon=\kappa^*\bar{\Upsilon}$. Observe that $\bar{\Upsilon}$ is closed as needed. Furthermore, there is a $(n-1)$-form on $\Sigma$, $\bar{\omega}$ such that $\omega=\pi^*_{\Sigma,P}\bar{\omega}$. Then,
$$\pi^*_{P,\Pi_P}\omega=\pi^*_{P,\Pi_P}\pi^*_{\Sigma,P}\bar{\omega}=\pi^*_{\Sigma,\Pi_P}\bar{\omega}=\kappa^*\bar{\pi}^*\bar{\omega}$$
and the second term of $F$ projects to $\bar{\pi}^*\bar{\omega}$.

The connection $\mathcal{A}$ on $P\to\Sigma$ splits the vector field $X$ in an horizontal and vertical part with respect to $P\to\Sigma$,
$$X=\mathrm{Hor}(X)+\mathrm{Ver}(X).$$
As $X$ is invariant, there exists a vector field $Y$ in $\Sigma$ such that for any $y\in P$,
$$(\mathrm{Hor}(X))_y=\hlift^y_{\mathcal{A}}(Y_s).$$
Moreover, from Proposition \ref{restr}, $Y\in\mathfrak{X}(\Sigma)$ is vertical with respect to $\Sigma\to M$ as $X$ is vertical with respect to $P\to M$.
The function $\xi:P\to\mathfrak{g}$ defined by $\xi(y)=\mathcal{A}((\mathrm{Ver}(Y))_y)$ is $\mathrm{Ad}$-equivariant and induces a section $\bar{\xi}$ in $\g\to\Sigma$.
To conclude this proof we shall show that for any $y\in P$, $(\theta_X)_y=\kappa^*((\theta_Y)_{\sigma}+(\theta_{\bar{\xi}})_{\mu})$. Indeed,
\begin{align*}
(\theta_X)_p=\llangle p, X\rrangle =&\llangle p, \mathrm{Hor}(X)\rrangle +\llangle p, \mathrm{Ver}(X)\rrangle =\llangle p, \hlift_{\mathcal{A}}(Y)\rrangle +\llangle p, \mathrm{Ver}(X)\rrangle\\
=&\kappa^*\llangle \sigma,Y \rrangle+\kappa^*\llangle \mu,[y,\mathcal{A}(\mathrm{Ver}(X))]_G \rrangle =\kappa^*(\theta_Y)_{\sigma}+\kappa^*(\theta_{\bar{\xi}})_{\mu},
\end{align*}
where $\llangle\cdot,\cdot\rrangle$ is used for the vector bundles $VP$, $V\Sigma$ and $\g$.
\end{proof}

\section{Connections} \label{SEC: Connections}
In this section we will study how both, the Ehresmann connection $\Lambda$ on $P\to M$ and the principal connection $\A$ on $P\to\Sigma$, are propagated to other bundles relevant to describe reduced Hamiltonian systems.

From \cite[Proposition 5.1]{someremarks}, stated at the end of Section \ref{preliminaries}, any Ehresmann  connection $\Lambda$ on $P\to M$ together with any Riemannian connection on $M$ provide an Ehresmann connection $\Lambda^{\Pi_P}$ on the bundle $\Pi_P\to M$. The horizontal lift in the local coordinates introduced in Section \ref{SEC: symmetries} is
\begin{align*}
\frac{\partial}{\partial x^i}\mapsto\frac{\partial}{\partial x^i}+\Lambda^a_i\frac{\partial}{\partial x^a}+\Lambda^{\alpha}_{i}\frac{\partial}{\partial y^{\alpha}}&+\left(-\frac{\partial \Lambda^b_i}{\partial x^a}p^j_b-\frac{\partial \Lambda^{\beta}_i}{\partial x^a}p^j_{\beta}+\Gamma^j_{ik}p^k_a-\Gamma^k_{ik}p^j_a\right)\frac{\partial}{\partial p^j_a}\\&+\left(-\frac{\partial \Lambda^{b}_i}{\partial y^{\alpha}}p^j_{b}-\frac{\partial \Lambda^{\beta}_i}{\partial y^{\alpha}}p^j_{\beta}+\Gamma^j_{ik}p^k_{\alpha}-\Gamma^k_{ik}p^j_{\alpha}\right)\frac{\partial}{\partial p^j_{\alpha}},
\end{align*}
where $\Lambda^{\alpha}_{i}$, $\Lambda^a_i$ are the coefficients of $\Lambda$ and $\Gamma^k_{ij}$ are the Christoffel symbols of the Riemannian connection.

Assume that $\Lambda$ is $G$-invariant in the sense that for every $v\in TM$
$$\hlift^{g\cdot y}_{\Lambda}(v)=T_y\Phi_g(\hlift^{y}_{\Lambda}(v)),$$
where $\hlift^{y}_{\Lambda}(v)$ denotes the horizontal lift of $v$ at $y\in P$ through $\Lambda$. Then there is an induced connection $\Lambda '$ on $\Sigma\to M$ such that for every $v\in TM$,
$$\hlift^{s}_{\Lambda '}(v)=T_y\pi_{\Sigma, P} (\hlift^y_{\Lambda}(v))$$
provided that $y\in P_s$. The local expression of the horizontal lift of $\Lambda '$ is $\frac{\partial}{\partial x^i}\mapsto\frac{\partial}{\partial x^i}+\Lambda^a_i(x^j,x^b)\frac{\partial}{\partial x^a}$. We will denote the induced connection $\Lambda$ via an abuse of notation. Furthermore, particularization of equation \eqref{conn.poly} to the induced connection on $\Sigma$ introduces a connection on $\Pi_{\Sigma}\to M$ whose horizontal lift is
$$\frac{\partial}{\partial x^i}\mapsto\frac{\partial}{\partial x^i}+\Lambda^a_i\frac{\partial}{\partial x^a}+\left(-\frac{\partial \Lambda^b_i}{\partial x^a}p^j_b+\Gamma^j_{ik}\sigma^k_a-\Gamma^k_{ik}\sigma^j_a\right)\frac{\partial}{\partial \sigma^j_a}.$$

\begin{definition} \label{compatible}
Let $\Lambda$ be an Ehresmann connection on $P\to M$, $\tilde{\Lambda}$ an Ehresmann connection on $\Sigma\to M$ and $\A$ a principal connection on $P\to \Sigma$, they are \textbf{compatible} if for every $y\in P$,
$\hlift^y_{\Lambda}=\hlift^y_{\A}\circ \hlift^s_{\tilde{\Lambda}}.$
\end{definition}
Observe that if the connections are compatible, $\tilde{\Lambda}$ is the induced connection on $\Sigma\to M$ by $\Lambda$ as
$$\hlift_{\Lambda '}^s=T_y\pi_{\Sigma, P}\circ \hlift^y_{\Lambda}=T_y\pi_{\Sigma, P}\circ \hlift^y_{\A}\circ \hlift^s_{\tilde{\Lambda}}=\hlift^s_{\tilde{\Lambda}}.$$
In local coordinates the compatibility condition is given by equations
\begin{equation}\label{compatibleP}
\Lambda^{\alpha}_i=A^{\alpha}_i+A^{\alpha}_b\Lambda^b_i
\end{equation}
\begin{equation}\label{compatibleSig}
\tilde{\Lambda}^b_i=\Lambda^b_i
\end{equation}

We also provide a connection in the vector bundle $TM\otimes\g^*\otimes\bigwedge^{n}T^*M \to \Sigma$. On one hand, the principal connection $\A$ on $P\to \Sigma$ provides a linear connection on $\g\to\Sigma$ and a dual connection on $\g^*\to\Sigma$. On the other hand, the Riemannian connection $\Gamma$ on $M$, thought as a linear connection on $TM\to M$, can be pulled back by $\pi_{M,\Sigma}$ to a linear connection on $\pi^*_{M,\Sigma}TM\to \Sigma$ whose horizontal lift is
\begin{align*}
\frac{\partial}{\partial x^i}&\mapsto\frac{\partial}{\partial x^i}+\Gamma^k_{ij}(x^i)\dot{x}^j\frac{\partial}{\partial \dot{x}^k},\\
\frac{\partial}{\partial x^a}&\mapsto\frac{\partial}{\partial x^a}.
\end{align*}
Furthermore, $\Gamma$ provides a linear connection on $\bigwedge^{n}T^*M\to M$ which can be pulled back to another one in $\pi^*_{M,\Sigma}\bigwedge^{n}T^*M\to \Sigma$. Then, the connection $\nabla^{\A,\Gamma}$ on $TM\otimes\g^*\otimes\bigwedge^{n}T^*M \to \Sigma$ obtained as the tensor product of the connections in each factor has the following horizontal lift:
\begin{align*}
\frac{\partial}{\partial x^i}&\mapsto\frac{\partial}{\partial x^i}+c^{\gamma}_{\beta\alpha}A^{\beta}_i\mu^j_{\gamma}\frac{\partial}{\partial \mu^j_{\alpha}}+\left(\Gamma^j_{ik}\mu^k_{\alpha}-\Gamma^k_{ik}\mu^j_{\alpha}\right)\frac{\partial}{\partial \mu^j_{\alpha}},\\
\frac{\partial}{\partial x^a}&\mapsto\frac{\partial}{\partial x^a}+c^{\gamma}_{\beta\alpha}A^{\beta}_a\mu^j_{\gamma}\frac{\partial}{\partial \mu^j_{\alpha}}.
\end{align*}
\begin{remark} Consider the pullback bundle
\begin{center}
\begin{tikzpicture}[scale=1.5] \label{pullbackjet}
\node (0) at (0,1) {$J^1\Sigma^*\oplus(TM\otimes\g^*\otimes\bigwedge^{n}T^*M)$};
\node (A) at (4,1) {$TM\otimes\g^*\otimes\bigwedge^{n}T^*M$};
\node (B) at (4,0) {$\Sigma$};
\node (C) at (0,0) {$J^1\Sigma^*$};
\draw[->,font=\scriptsize,>=angle 90]
(0) edge node[above]{} (A)
(0) edge node[right]{$\pi_1$} (C)
(C) edge node[above]{$\pi_{\Sigma,J^1\Sigma^*}$} (B)
(A) edge node[right]{} (B);
\end{tikzpicture}
\end{center}
The pullback connection $\pi^*_{\Sigma, J^1{\Sigma}^*}\nabla^{\A,\Gamma}$ is defined on $\pi_1:J^1{\Sigma}^*\oplus(TM\otimes\g^*\otimes\bigwedge^{n}T^*M) \to J^1{\Sigma}^*$. We will denote this connection by $\nabla^{\A,\Gamma}$ as well. Similarly, the pullback connection $\pi^*_{\Sigma, \Pi_{\Sigma}}\nabla^{\A,\Gamma}$ is defined on $\Pi_{\Sigma}\oplus(TM\otimes\g^*\otimes\bigwedge^{n}T^*M) \to \Pi_{\Sigma}$ and denoted $\nabla^{\A,\Gamma}$.

\end{remark}
Observe that using $\Lambda$ to lift a vector from $M$ to $\Sigma$ as well as the lift provided by $\nabla^{\A,\Gamma}$ defines an Ehresmann connection $\Lambda^{\A,\Gamma}$ on $TM\otimes\g^*\otimes\bigwedge^{n}T^*M \to M$.
\begin{proposition}\label{econnectiong}
Let $\Lambda$, $\tilde{\Lambda}$ and $\A$ be compatible connections on $P\to M$, $\Sigma\to M$ and $P\to \Sigma$, the connection $\tilde{\Lambda}^{\A,\Gamma}$ on $TM\otimes\g^*\otimes\bigwedge^{n}T^*M \to M$ does not depend on $\A$.
\end{proposition}
\begin{proof}
We will obtain the horizontal lift of a basic vector $\frac{\partial}{\partial x^i}$ to $\mu \in TM\otimes\g^*\otimes\bigwedge^{n}T^*M \to M$ with coordinates $(x^i,x^a,\mu^j_{\alpha})$. Firstly, $\frac{\partial}{\partial x^i}$ lifts to $\frac{\partial}{\partial x^i}+\tilde{\Lambda}^a_ i\frac{\partial}{\partial x^a}\in T_s\Sigma$, where $s=(x^i.x^a)$. The later is then lifted via $\nabla^{\A,\Gamma}$ to
\begin{align*}
\frac{\partial}{\partial x^i}+\tilde{\Lambda}^a_i\frac{\partial}{\partial x^a}+\left(\mu^j_{\gamma}c^{\gamma}_{\beta\alpha}(A^{\beta}_i+A^{\beta}_a\tilde{\Lambda}^a_i)+\Gamma^j_{ik}\mu^k_{\alpha}-\Gamma^k_{ik}\mu^j_{\alpha}\right)\frac{\partial}{\partial\mu^j_{\alpha}}\\
=\frac{\partial}{\partial x^i}+\Lambda^a_i\frac{\partial}{\partial x^a}+\left(\mu^j_{\gamma}c^{\gamma}_{\beta\alpha}{\Lambda}^{\beta}_i+\Gamma^j_{ik}\mu^k_{\alpha}-\Gamma^k_{ik}\mu^j_{\alpha}\right)\frac{\partial}{\partial\mu^j_{\alpha}},
\end{align*}
where the compatibility equations \eqref{compatibleP} and \eqref{compatibleSig} have been used.
\end{proof}

Finally, an Ehresmann connection $\Lambda^{\Pi_{\Sigma}}\oplus\Lambda^{\A,\Gamma}$ is defined on the reduced polysymplectic space $\Pi_{\Sigma}\oplus (TM\otimes\g^*\otimes\bigwedge^{n}T^*M) \to M$.

\section[Poisson Bracket on the Reduced Polysymplectic Bundle]{Poisson Bracket on $\Pi_P/G$} \label{SEC:hor bracket}

This section establishes a canonical \emph{Poisson bracket} on $\Pi _P/G\simeq \Pi_{\Sigma}\oplus (TM\otimes\g^*\otimes\bigwedge^{n}T^{\ast}M)$ compatible with the Poisson structure on $\Pi _P$ and the choice of connections above. This bracket consists of three components.

\subsection{Poisson Forms on the Reduced Polysymplectic Bundle}
The first component that we analyze is attached to the polysymplectic bundle $\Pi_\Sigma$. This bundle is equipped with a Poisson structure in the sense of \cite{PoissonForms}. However, this bracket will not appear as one of the ingredients when the reduction of the bracket of $\Pi_P$ is performed. We require a structure that takes into account the fact that $\Pi_\Sigma$ is twisted with the vector bundle $TM\otimes\g^*\otimes\bigwedge^{n}T^*M$ and will need the connection $\nabla ^{\mathcal{A},\Gamma}$ constructed in Section \ref{SEC: Connections}. We define this structure by introducing first the corresponding notion of Poisson forms through a series of preliminary definitions.

\begin{definition}
An $r$-form $f$ on $J^1\Sigma^*\oplus(TM\otimes\g^*\otimes\bigwedge^{n}T^*M)$ is said to be $J^1\Sigma^*$-horizontal if $i_uf=0$ for any vertical tangent vector $u$ with respect to $\pi_1:J^1\Sigma^*\oplus(TM\otimes\g^*\otimes\bigwedge^{n}T^*M)\to J^1\Sigma^*$
\end{definition}
In local coordinates these forms are a linear combination of basic $r$-forms on $J^1\Sigma^*$ with coefficients in $\mathcal{C}^\infty(J^1\Sigma^*\oplus(TM\otimes\g^*\otimes\bigwedge^{n}T^*M))$. Alternatively, a $J^1\Sigma^*$-horizontal $r$-form on $J^1\Sigma^*\oplus(TM\otimes\g^*\otimes\bigwedge^{n}T^*M)$ can be defined as a section in the pullback bundle

\begin{center}
\begin{tikzpicture}[scale=1.5] \label{jsigmahor}
\node (0) at (0,1) {$\pi_1^*\left(\bigwedge^rT^*(J^1\Sigma^*)\right)$};
\node (A) at (4,1) {$\bigwedge^rT^*(J^1\Sigma^*)$};
\node (B) at (4,0) {$J^1\Sigma^*$};
\node (C) at (0,0) {$J^1\Sigma^*\oplus(TM\otimes\g^*\otimes\bigwedge^{n}T^*M)$};
\draw[->,font=\scriptsize,>=angle 90]
(0) edge node[above]{} (A)
(0) edge node[right]{} (C)
(C) edge node[above]{$\pi_1$} (B)
(A) edge node[right]{} (B);
\end{tikzpicture}
\end{center}
In general, the exterior derivative of a $J^1\Sigma^*$-horizontal $r$-form is not a $J^1\Sigma^*$-horizontal $(r+1)$-form. Therefore, we shall use instead the horizontal differential $d^{h,J^1\Sigma^*}$ of the connection $\nabla^{\A,\Gamma}$ in $\pi_1:J^1\Sigma^*\oplus(TM\otimes\g^*\otimes\bigwedge^{n}T^*M)\to J^1\Sigma^*$.
\begin{proposition}
Let $f$ be a $J^1\Sigma^*$-horizontal $r$-form, then $d^{h,J^1\Sigma^*}f$ is a $J^1\Sigma^*$-horizontal $(r+1)$-form.
%For any $J^1\Sigma^*$-horizontal $r$-form $f$, $d^{h,J^1\Sigma^*}f$ is a $J^1\Sigma^*$-horizontal $(r+1)$-form.
\end{proposition}
\begin{proof}
Let $f$ be a $0$-form, that is, a differentiable function on $J^1\Sigma^*\oplus(TM\otimes\g^*\otimes\bigwedge^{n}T^*M)$. The horizontal differential of $f$ is
\begin{multline}\label{dh0forma}
d^{h,J^1\Sigma^*}f= \left(\frac{\partial f}{\partial x^i}+c^{\gamma}_{\beta\alpha}A^{\beta}_i\mu^j_{\gamma}\frac{\partial f}{\partial \mu^j_{\alpha}}+\left(\Gamma^j_{ik}\mu^k_{\alpha}-\Gamma^k_{ik}\mu^j_{\alpha}\right)\frac{\partial f}{\partial \mu^j_{\alpha}}\right)dx^i \\
+\left(\frac{\partial f}{\partial x^a}+c^{\gamma}_{\beta\alpha}A^{\beta}_a\mu^j_{\gamma}\frac{\partial f}{\partial \mu^j_{\alpha}}\right)dx^a+\frac{\partial f}{\partial \sigma^i_a}d\sigma^i_a+\frac{\partial f}{\partial \tsigma}d\tsigma,
\end{multline}
which is a $J^1\Sigma^*$-horizontal $1$-form, as it has no terms proportional to any $d\mu^i_{\alpha}$. Suppose now that $f$ is a $J^1\Sigma^*$-horizontal $r$-form, its horizontal differential is obtained by horizontal differentiation of its coefficients. Thus, there will be no terms in $d^{h,J^1\Sigma^*}f$ including the wedge with any $d\mu^i_{\alpha}$.
\end{proof}

\begin{definition} A $J^1\Sigma^*$-horizontal $s$-multivector field on $J^1\Sigma^*\oplus(TM\otimes\g^*\otimes\bigwedge^{n}T^*M)$ is a section of $$\pi^*_1\left(\bigwedge^sT(J^1\Sigma^*)\right)\to J^1\Sigma^*\oplus(TM\otimes\g^*\otimes\bigwedge^{n}T^*M).$$
\end{definition}
In local coordinates, these $J^1\Sigma^*$-horizontal $s$-multivector fields are a linear combination of $s$-multivector fields on $J^1\Sigma^*$ with coefficients in $\mathcal{C}^\infty(J^1\Sigma^*\oplus(TM\otimes\g^*\otimes\bigwedge^{n}T^*M))$.
\begin{definition}
An $r$-form $f$ on $J^1\Sigma^*\oplus(TM\otimes\g^*\otimes\bigwedge^{n}T^*M)$ is said to be horizontal if $i_uf=0$ for any vertical tangent vector $u$ with respect to $J^1\Sigma^*\oplus(TM\otimes\g^*\otimes\bigwedge^{n}T^*M)\to M$
\end{definition}
Observe that the definition of horizontal and $J^1\Sigma^*$-horizontal forms only differs in the choice of bundle with respect to which $u$ is a vertical tangent vector. In fact, any horizontal form is  $J^1\Sigma^*$-horizontal.
\begin{definition}
An $s$-multivector $\chi$ on $J^1\Sigma^*\oplus(TM\otimes\g^*\otimes\bigwedge^{n}T^*M)$ is said to be vertical if its contraction $i_{\chi}f$ with any horizontal $s$-form $f$ on $J^1\Sigma^*\oplus(TM\otimes\g^*\otimes\bigwedge^{n}T^*M)\to M$ vanishes.
\end{definition}
\begin{definition}
A horizontal $r$-form $f$ on $J^1\Sigma^*\oplus(TM\otimes\g^*\otimes\bigwedge^{n}T^*M)$ is a Poisson form if there is a vertical $J^1\Sigma^*$-horizontal $(n-r)$-multivector field $\chi_f$ on $J^1\Sigma^*\oplus(TM\otimes\g^*\otimes\bigwedge^{n}T^*M)$ such that
\begin{equation} \label{eq:condicionpoisson}
i_{\chi_f}\Omega_{\Sigma}=d^{h,J^1\Sigma^*}f,
\end{equation}
where $\Omega_{\Sigma}$ is the multisymplectic form on $J^1\Sigma^*$. Given a Poisson $r$-form $f$ and a Poisson $s$-form $h$ on $J^1\Sigma^*\oplus(TM\otimes\g^*\otimes\bigwedge^{n}T^*M)$, their horizontal Poisson bracket is
\begin{equation}
\{f,h\}_{\Sigma}=(-1)^{r(s-1)}i_{\chi_h}i_{\chi_f}\Omega_{\Sigma}
\end{equation}
\end{definition}
In the following, some results are exhibited to justify that this is the first component of the searched Poisson bracket in $\Pi_{\Sigma}\oplus(TM\otimes\g^*\otimes\bigwedge^{n}T^*M)$.
\begin{proposition}\label{0-forma}
A function $f:J^1\Sigma^*\oplus(TM\otimes\g^*\otimes\bigwedge^{n}T^*M)\to\R$ is a Poisson $0$-form if and only if $f$ is projectable to $\Pi_{\Sigma}\oplus(TM\otimes\g^*\otimes\bigwedge^{n}T^*M)$.
\end{proposition}
\begin{proof}
The local expression for $d^{h,J^1\Sigma^*}f$ is given in equation \eqref{dh0forma} and
$\Omega_{\Sigma}=dx^a\wedge d\sigma^i_a\wedge \vol_i-d\tsigma\wedge \vol,$
where $\vol_i=i_{\partial/\partial x^i}\vol$. It can easily be checked that the $n$-multivector field
\begin{multline}\label{chi0forma}
\chi= -\left(\frac{\partial f}{\partial x^i}+c^{\gamma}_{\beta\alpha}A^{\beta}_i\mu^j_{\gamma}\frac{\partial f}{\partial \mu^j_{\alpha}}+\left(\Gamma^j_{ik}\mu^k_{\alpha}-\Gamma^k_{ik}\mu^j_{\alpha}\right)\frac{\partial f}{\partial \mu^j_{\alpha}}\right)\frac{\partial}{\partial \tsigma}\wedge\vol_i^*\\
-\left(\frac{\partial f}{\partial x^a}+c^{\gamma}_{\beta\alpha}A^{\beta}_a\mu^j_{\gamma}\frac{\partial f}{\partial \mu^j_{\alpha}}\right)\frac{\partial}{\partial \sigma^i_a}\wedge\vol_i^*+\frac{\partial f}{\partial \sigma^i_a}\frac{\partial}{\partial x^a}\wedge\vol_i^*+\frac{\partial f}{\partial \tsigma}\vol^*,
\end{multline}
satisfies the condition \eqref{eq:condicionpoisson}. If $f$ is projectable, $\frac{\partial f}{\partial \tsigma}=0$. Thus, $\chi$ is a vertical $n$-multivector and $f$ is Poisson. Conversely, if $f$ is Poisson with vertical $J^1\Sigma^*$-horizontal $n$-multivector $\chi_f$ satisfying the condition \eqref{eq:condicionpoisson}, $\chi_f$ does not have a term proportional to $\vol^*$. Then, $i_{\chi_f}\Omega_{\Sigma}$ does not have a term proportional to $d\tsigma$. From condition \eqref{eq:condicionpoisson}, the same is true for $d^{h,J^1\Sigma^*}f$ and $\frac{\partial f}{\partial \tsigma}=0$.
\end{proof}
\begin{proposition} If a horizontal $r$-form f, $r>0$, on $J^1\Sigma^*\oplus(TM\otimes\g^*\otimes\bigwedge^{n}T^*M)$ is Poisson, then it is projectable to $\Pi_{\Sigma}\oplus(TM\otimes\g^*\otimes\bigwedge^{n}T^*M)$.
\end{proposition}
\begin{proof}
In local coordinates, $f=f^{i_1\dots i_s}\vol_{i_1\dots i_s},$ where $s=n-r$ and $\vol_{i_1\dots i_s}=i_{\partial/\partial x^{i_1}}\dots i_{\partial/\partial x^{i_s}}\vol$. Consequently,
\begin{align*}
d^{h,J^1\Sigma^*}f=&\left(\frac{\partial f^{i_2\dots i_s,j}}{\partial x^j}+\left(c^{\gamma}_{\beta\alpha}A^{\beta}_j\mu^k_{\gamma}+(\Gamma^k_{jl}\mu^l_{\alpha}-\Gamma^l_{jl}\mu^k_{\alpha})\right)\frac{\partial f^{i_2\dots i_s,j}}{\partial \mu^k_{\alpha}}\right)\vol_{i_2\dots i_s}\\
&+\left(\frac{\partial f^{i_1\dots i_s}}{\partial x^a}+c^{\gamma}_{\beta\alpha}A^{\beta}_a\mu^k_{\gamma}\frac{\partial f^{i_1\dots i_s}}{\partial \mu^k_{\alpha}}\right)dx^a\wedge\vol_{i_1\dots i_s}\\
&+\frac{\partial f^{i_1\dots i_s}}{\partial \sigma^i_a}d\sigma^i_a\wedge\vol_{i_1\dots i_s}+\frac{\partial f^{i_1\dots i_s}}{\partial \tsigma}d\tsigma\wedge\vol_{i_1\dots i_s}
\end{align*}
Let $\chi_f$ be the $s$-multivertor field such that $i_{\chi_f}\Omega_{\Sigma}=d^{h,J^1\Sigma^*}f$. Then $\chi_f$ does not contain terms with component $\frac{\partial}{\partial x^{i_1}}\wedge\dots\wedge\frac{\partial}{\partial x^{i_s}}$ as that would imply that $i_{\chi_f}\Omega_{\Sigma}$ has terms with component $dx^a\wedge d\sigma^i_a\wedge\vol_{i,i_1\dots i_s}$ not present in $d^{h,J^1\Sigma^*}f$. Consequently, $d^{h,J^1\Sigma^*}f$ does not have any term with component $d\tsigma\wedge\vol_{i_1\dots i_s}$, $\frac{\partial f^{i_1\dots i_s}}{\partial \tsigma}=0$ and $f$ is projectable.
\end{proof}

From this last proposition we can consider that Poisson forms on $J^1{\Sigma}^*\oplus(TM\otimes\g^*\otimes\bigwedge^{n}T^*M)$ are defined on $\Pi_{\Sigma}\oplus(TM\otimes\g^*\otimes\bigwedge^{n}T^*M)$. As the bracket of Poisson forms is also a Poisson form, the horizontal bracket $\{f,g\}_{\Sigma}$ will be hereafter considered a bracket of Poisson forms on $\Pi_{\Sigma}\oplus(TM\otimes\g^*\otimes\bigwedge^{n}T^*M)$. We should now focus on the bracket of $(n-1)$-forms and functions, which is of special interest to study the dynamics of physical systems.

\begin{proposition}\label{n-1.formas}
Any Poisson $(n-1)$-form on $\Pi_{\Sigma}\oplus(TM\otimes\g^*\otimes\bigwedge^{n}T^*M)$ can be written as
$$f=\theta_Y+\hat{\pi}^*_2\varpi+\bar{\Upsilon},$$
where $Y:TM\otimes\g^*\otimes\bigwedge^{n}T^*M\to V\Sigma$ is a bundle morphism covering $\mathrm{Id}_\Sigma$, $\varpi$ is an horizontal $(n-1)$-form on $TM\otimes\g^*\otimes\bigwedge^{n}T^*M\to M$, $\bar{\Upsilon}$ is a closed horizontal $(n-1)$-form on $\Pi_{\Sigma}\oplus (TM\otimes\g^*\otimes\bigwedge^{n}T^*M),$ and $\hat{\pi}_2=\Pi_{\Sigma}\oplus(TM\otimes\g^*\otimes\bigwedge^{n}T^*M)\to TM\otimes\g^*\otimes\bigwedge^{n}T^*M$ is the projection to the second factor.
\end{proposition}
\begin{proof}
In local coordinates, $f=f^i\vol_i$ and
\begin{multline*}
d^{h,J^1\Sigma^*}f=\left(\frac{\partial f^i}{\partial x^i}+\left(c^{\gamma}_{\beta\alpha}A^{\beta}_i\mu^j_{\gamma}+\left(\Gamma^j_{ik}\mu^k_{\alpha}-\Gamma^k_{ik}\mu^j_{\alpha}\right)\right)\frac{\partial f^i}{\partial \mu^j_{\alpha}}\right)\vol\\
+\left(\frac{\partial f^i}{\partial x^a}+c^{\gamma}_{\beta\alpha}A^{\beta}_a\mu^k_{\gamma}\frac{\partial f^i}{\partial \mu^k_{\alpha}}\right)dx^a\wedge\vol_i
+\frac{\partial f^i}{\partial \sigma^j_a}d\sigma^j_a\wedge\vol_i.
\end{multline*}
As $f$ is Poisson, there is a $J^1\Sigma^*$-horizontal vector field
$$\chi_f=Y^a\frac{\partial}{\partial x^a}+Y^i_a\frac{\partial}{\partial \sigma^i_a}+W\frac{\partial}{\partial \tsigma}$$
such that $i_{\chi_f}\Omega_{\Sigma}=d^{h,J^1\Sigma^*}f$. Since
$$i_{\chi_f}\Omega_{\Sigma}=Y^ad\sigma^i_a\wedge\vol_i-Y^i_adx^a\wedge\vol_i-W\vol,$$
we obtain the following conditions for $\chi_f$:
\begin{align*}
W=-\left(\frac{\partial f^i}{\partial x^i}+\left(c^{\gamma}_{\beta\alpha}A^{\beta}_i\mu^j_{\gamma}+(\Gamma^j_{ik}\mu^k_{\alpha}-\Gamma^k_{ik}\mu^j_{\alpha})\right)\frac{\partial f^i}{\partial \mu^j_{\alpha}}\right),\\
Y^i_a=-\frac{\partial f^i}{\partial x^a}-c^{\gamma}_{\beta\alpha}A^{\beta}_a\mu^k_{\gamma}\frac{\partial f^i}{\partial \mu^k_{\alpha}},\phantom{espacio}
\delta^i_jY^a=\frac{\partial f^i}{\partial \sigma^j_a}.
\end{align*}
In conclusion, the general expression of a Poisson $(n-1)$-form is $f=(\sigma^i_aY^a+g^i)\vol_i$ for any functions $Y^a$, $g^i$ on $TM\otimes\g^*\otimes\bigwedge^{n}T^*M$. This allows to identify the first two terms in the Proposition and the result follows since Poisson forms are defined up to a closed form.
\end{proof}

From Theorem \ref{red.n-1.forms}, the projection of a $G$-invariant Poisson $(n-1)$-form on $\Pi_P$ is a Poisson $(n-1)$-form on $\Pi_{\Sigma}\oplus (TM\otimes\g^*\otimes\bigwedge^{n}T^*M)$ of a particular kind where $Y:TM\otimes\g^*\otimes\bigwedge^{n}T^*M\to V\Sigma$ projects to a vertical field $Y:\Sigma\to V\Sigma$ and $\varpi$, $(n-1)$-form in $TM\otimes\g^*\otimes\bigwedge^{n}T^*M\to M$,
can be written as:
$$\varpi=\theta_{\bar{\xi}}+\pi^*_{\Sigma,TM\otimes\g^*\otimes\bigwedge^{n}T^*M}\bar{\omega},$$
where $\bar{\xi}$ is a section in $\g\to \Sigma$ and $\bar{\omega}$ is an horizontal $(n-1)$-form on $\Sigma$. We will say such Poisson $(n-1)$-forms are affine.

The horizontal bracket of a Poisson $(n-1)$-form $f$ on $\Pi_{\Sigma}\oplus(TM\otimes\g^*\otimes\bigwedge^{n}T^*M)$ and a Poisson function $h$ can be easily expressed in local coordinates as
\begin{equation}
\{f,h\}_{\Sigma}=\left(\frac{\partial h}{\partial x^a}+c^{\gamma}_{\beta\alpha}A^{\beta}_a\mu^j_{\gamma}\frac{\partial h}{\partial \mu^j_{\alpha}}\right)\frac{\partial f^i}{\partial \sigma^i_a}-\left(\frac{\partial f^i}{\partial x^a}+c^{\gamma}_{\beta\alpha}A^{\beta}_a\mu^j_{\gamma}\frac{\partial f^i}{\partial \mu^j_{\alpha}}\right)\frac{\partial h}{\partial \sigma^i_a}
\end{equation}
from the local expressions in the proof of Propositions \ref{0-forma} and \ref{n-1.formas}. This expression is similar to the usual Poisson bracket defined for forms on $\Pi_{\Sigma}$ replacing partial derivatives $\partial/\partial x^a$ by their covariant ones through $\nabla^{\A,\Gamma}$ since the forms considered here are defined on $\Pi_{\Sigma}\oplus(TM\otimes\g^*\otimes\bigwedge^{n}T^*M)$ and not just on $\Pi_{\Sigma}$.

\subsection{The Construction of the Bracket}

For any function $h\in\mathcal{C}^{\infty}(\Pi_{\Sigma}\oplus(TM\otimes\g^*\otimes\bigwedge^{n}T^*M))$ there are two fiber derivatives. On one hand
\begin{equation*}
\frac{\delta h}{\delta\sigma}:\Pi_{\Sigma}\oplus (TM\otimes\g^*\otimes\bigwedge^{n}T^*M) \to T^*M\otimes V\Sigma\otimes\bigwedge^{n}TM
\end{equation*}
defined by
\begin{equation*}
\frac{\delta h}{\delta\sigma}(\sigma\oplus\mu)(\tau)=\frac{d}{d\varepsilon}\bigg|_{\varepsilon=0}h((\sigma+\varepsilon\tau)\oplus\mu),
\end{equation*}
whose expression in local coordinates is
\begin{equation}
\frac{\partial h}{\partial \sigma^i_a} dx^i\otimes \frac{\partial }{\partial x^a}\otimes \vol.
\end{equation}
On the other hand
\begin{equation*}
\frac{\delta h}{\delta\mu}:\Pi_{\Sigma}\oplus (TM\otimes\g^*\otimes\bigwedge^{n}T^*M) \to T^*M\otimes\g\otimes\bigwedge^{n}TM
\end{equation*}
given by
\begin{equation*}
\frac{\delta h}{\delta\mu}(\sigma\oplus\mu)(\nu)=\frac{d}{d\varepsilon}\bigg|_{\varepsilon=0}h(\sigma\oplus(\mu+\varepsilon\nu)),
\end{equation*}
whose local expression is
\begin{equation}
\frac{\partial h}{\partial \mu^i_{\alpha}} dx^i\otimes \tilde{B}_{\alpha}\otimes \vol.
\end{equation}
With this fiber derivative we define a second component of the bracket in $\Pi_P/G$ attached to $TM\otimes\g^*\otimes\bigwedge^{n}T^*M$. This is a Lie-Poisson bracket in Filed Theories as the one in \cite{someremarks}.

\begin{definition}
Given a function $h\in\mathcal{C}^{\infty}(\Pi_{\Sigma}\oplus(TM\otimes\g^*\otimes\bigwedge^{n}T^*M))$ and a section $\bar{\xi}$ of $\g\to \Sigma$ we define
\begin{equation}
\label{LiePoissonBra}
\{\bar{\xi},h\}_{LP}=-\left\langle\mu,\left[\bar{\xi},\frac{\delta h}{\delta\mu}\right]\right\rangle,\qquad \mu \in TM\otimes\g^*\otimes\bigwedge^{n}T^*M.
\end{equation}
\end{definition}

Finally, our bracket in $\Pi_{\Sigma}\oplus(TM\otimes\g^*\otimes\bigwedge^{n}T^*M)$ needs a deformation term defined by the curvature. More precisely,
\begin{align*}
\bar{\B}: T\Sigma \otimes T^*M\otimes T\Sigma \otimes\bigwedge^{n}TM &\to T^*M\otimes  \g \otimes\bigwedge^{n}TM \\
\left(\frac{\partial }{\partial x^a} , dx^i\otimes\frac{\partial }{\partial x^b}\otimes\vol^*\right) &\mapsto dx^i\otimes\tilde{\B}\left(\frac{\partial }{\partial x^a},\frac{\partial }{\partial x^b}\right)\otimes\vol^*,
\end{align*}
where $\tilde{\B}$ is  the $2$-form defined on $\Sigma$ by the curvature form $\B$ of the connection $\A$ in $P\to \Sigma$

Once we have the bracket steaming from $\Pi_\Sigma$, the Lie Poisson bracket and the deformation we can define the full bracket in the reduced space as follows.

\begin{definition}
Let $f$ be an affine Poisson $(n-1)$-form $\Pi_{\Sigma}\oplus(TM\otimes\g^*\otimes\bigwedge^{n}T^*M)$ and $h$ a Poisson function on $\Pi_{\Sigma}\oplus(TM\otimes\g^*\otimes\bigwedge^{n}T^*M)$, their Poisson bracket is
\begin{equation} \label{reduced bracket}
\{f,h\}=\{f,h\}_{\Sigma}-\left\langle\mu,\left[\bar{\xi},\frac{\delta h}{\delta\mu}\right]\right\rangle -\left\langle\mu,\bar{\B}\left(Y,\frac{\delta h}{\delta\sigma}\right)\right\rangle,
\end{equation}
where $Y$ and $\bar{\xi}$ are respectively sections of $T\Sigma\to\Sigma$ and $\g\to\Sigma$ that define $f$ as in Theorem \ref{red.n-1.forms}.
\end{definition}

\section{Poisson--Poincar\'e Reduction}\label{SEC:reduction}
In this Section Poisson-Poincar\'e reduction is presented in the multisymplectic and polysymplectic formalism. The reduced dynamics are described by the bracket constructed in Section \ref{SEC:hor bracket}.
\begin{theorem} \label{kappa.is.poisson}
The projection map%
\begin{equation*}
\kappa:\Pi_{P}\longrightarrow\Pi_{P}/G\cong\Pi_{\Sigma}\oplus (TM\otimes\g^*\otimes\bigwedge^{n-1}T^*M)
\end{equation*}
is a Poisson map in the sense that for any $G$-invariant Poisson $(n-1)$-form $F$ and any $G$-invariant Poisson function $H$ on $\Pi_P$
\begin{equation}
\{F,H\}=\kappa^*\{f,h\},
\end{equation}
where the bracket on the left hand side is defined by \eqref{Poisson bracket}, $f$ is a Poisson $(n-1)$-form on $\Pi_{\Sigma}\oplus (TM\otimes\g^*\otimes\bigwedge^{n-1}T^*M)$ such that $\kappa^*f=F$, $h$ is a function such that $\kappa^*h=H$, and the bracket on the right hand side is defined by \eqref{reduced bracket}.
\end{theorem}
\begin{proof}
In the coordinate system introduced in $\Pi_P$, we have the local expression
\begin{align*}
\{F,H\}=\frac{\partial F^i}{\partial x^a}\frac{\partial H}{\partial p^i_a}-\frac{\partial F^i}{\partial p^i_a}\frac{\partial H}{\partial x^a}+\frac{\partial F^i}{\partial y^{\alpha}}\frac{\partial H}{\partial p^i_{\alpha}}-\frac{\partial F^i}{\partial p^i_{\alpha}}\frac{\partial H}{\partial y^{\alpha}}
\end{align*}
As $F$ and $H$ are $G$-invariant, using the differential of projection $\kappa$ obtained in Proposition \ref{dif.kappa} gives a function $r$ on $\Pi_{\Sigma}\oplus (TM\otimes\g^*\otimes\bigwedge^{n-1}T^*M)$
\begin{align*}
r=&\frac{\partial f^i}{\partial x^a}\frac{\partial h}{\partial \sigma^i_a}+\mu^j_{\gamma}\frac{\partial A^{\gamma}_b}{\partial x^a}\frac{\partial f^i}{\partial \sigma^j_b}
\frac{\partial h}{\partial \sigma^i_a}-\frac{\partial f^i}{\partial \sigma^i_a}\frac{\partial h}{\partial x^a}-\frac{\partial f^i}{\partial \sigma^i_a}\mu^j_{\gamma}\frac{\partial A^{\gamma}_b}{\partial x^a}\frac{\partial h}{\partial \sigma^j_b}\\
&+\left(-\frac{1}{2} \mu^j_{\gamma}c^{\gamma}_{\beta\alpha}A^{\beta}_b\frac{\partial f^i}{\partial\sigma^j_b}-\frac{1}{2} \mu^j_{\gamma}c^{\gamma}_{\beta\alpha}\frac{\partial f^i}{\partial\mu^j_{\beta}}\right)\left(\frac{\partial h}{\partial \mu^i_{\alpha}}+A^{\alpha}_a\frac{\partial h}{\partial \sigma^i_a}\right)\\
&-\left(\frac{\partial f^i}{\partial \mu^i_{\alpha}}+A^{\alpha}_a\frac{\partial f^i}{\partial \sigma^i_a}\right)\left(-\frac{1}{2} \mu^j_{\gamma}c^{\gamma}_{\beta\alpha}A^{\beta}_b\frac{\partial h}{\partial\sigma^j_b}-\frac{1}{2} \mu^j_{\gamma}c^{\gamma}_{\beta\alpha}\frac{\partial h}{\partial\mu^j_{\beta}}\right),
\end{align*}
such that $\{F,H\}=\kappa^*r$. We shall see that $r=\{f,h\}$ as desired. Since $f$ is the projection of $F$, from Theorem \ref{red.n-1.forms}, $f$ is an affine Poisson form and $\partial f^i/\partial \sigma^j_a=Y^a\delta^i_j$ and $\partial f^i/\partial \mu^j_{\alpha}=\xi^{\alpha}\delta^i_j$. Therefore,
\begin{align*}
r=&\frac{\partial f^i}{\partial x^a}\frac{\partial h}{\partial \sigma^i_a}-\frac{\partial f^i}{\partial \sigma^i_a}\frac{\partial h}{\partial x^a}+\mu^j_{\gamma}\left(\frac{\partial A^{\gamma}_b}{\partial x^a}-\frac{\partial A^{\gamma}_a}{\partial x^b}\right)\frac{\partial f^i}{\partial \sigma^j_b}\frac{\partial h}{\partial \sigma^i_a}+\mu^j_{\gamma}c^{\gamma}_{\beta\alpha}A^{\beta}_aA^{\alpha}_b\frac{\partial h}{\partial\sigma^i_a}\frac{\partial f^i}{\partial\sigma^j_b}\\
&+
\mu^j_{\gamma}c^{\gamma}_{\beta\alpha}\frac{\partial h}{\partial\mu^j_{\beta}}\frac{\partial f^i}{\partial\mu^i_{\alpha}}+\mu^j_{\gamma}c^{\gamma}_{\beta\alpha}A^{\alpha}_b\frac{\partial h}{\partial\mu^j_{\beta}}\frac{\partial f^i}{\partial\sigma^j_b}+
\mu^j_{\gamma}c^{\gamma}_{\beta\alpha}A^{\beta}_a\frac{\partial h}{\partial\sigma^i_a}\frac{\partial f^i}{\partial\mu^j_{\alpha}}\\
=&\left(\frac{\partial f^i}{\partial x^a}+\mu^j_{\gamma}c^{\gamma}_{\beta\alpha}A^{\beta}_a\frac{\partial f^i}{\partial\mu^j_{\alpha}}\right)\frac{\partial h}{\partial \sigma^i_a}
-\frac{\partial f^i}{\partial \sigma^i_a}\left(\frac{\partial h}{\partial x^a}+\mu^j_{\gamma}c^{\gamma}_{\beta\alpha}A^{\beta}_a\frac{\partial h}{\partial\mu^j_{\alpha}}\right)\\
&+\mu^j_{\gamma}\left(\frac{\partial A^{\gamma}_b}{\partial x^a}-\frac{\partial A^{\gamma}_a}{\partial x^b}+c^{\gamma}_{\beta\alpha}A^{\beta}_aA^{\alpha}_b\right)\frac{\partial f^i}{\partial \sigma^j_b}\frac{\partial h}{\partial \sigma^i_a}
+\mu^j_{\gamma}c^{\gamma}_{\beta\alpha}\frac{\partial h}{\partial\mu^j_{\beta}}\frac{\partial f^i}{\partial\mu^i_{\alpha}},
\end{align*}
where the three terms of $\{f,h\}$ can be identified.
\end{proof}

Consider the following map related with the lift of the Ehresmann connection $\Lambda$ on $\Sigma\to M$:
\begin{align*}
\bar{\Lambda}:\Sigma &\to T^*M\otimes T\Sigma \otimes\bigwedge^{n}TM \\
s &\mapsto dx^i\otimes\left(\delta^j_i\frac{\partial}{\partial x^j}+\Lambda_i^a(s)\frac{\partial}{\partial x^a}\right)\otimes\vol^*.
\end{align*}
\begin{lemma} \label{red.dh}
Let $F$ be a $G$-invariant Poisson $(n-1)$-form on $\Pi_P$ and $f$ the affine Poisson $(n-1)$-form on $\Pi_{\Sigma}\oplus (TM\otimes\g^*\otimes\bigwedge^{n-1}T^*M)$ such that $\kappa^*f=F$ as in Theorem \ref{red.n-1.forms}, then
\begin{equation}
d^hF=\kappa^*\left(d^hf-\left\langle\mu, \bar{\B}(Y,\bar{\Lambda})\right\rangle\vol\right),
\end{equation}
where $d^hF$ is the horizontal derivative with respect to the Ehresmann connection $\Lambda^{\Pi_P}$ on $\Pi_P\to M$ and $d^hf$ is the horizontal derivative with respect to the Ehresmann connection $\Lambda^{\Pi_{\Sigma}}\oplus\Lambda^{\A,\Gamma}.$
\end{lemma}
\begin{proof}
The local expression of the horizontal derivative of an $(n-1)$-form $F=F^i\vol_i$ is
\begin{align*}
d^hF=&\left(\frac{\partial F^i}{\partial x^i}+\Lambda^a_i\frac{\partial F^i}{\partial x^a}+\Lambda^{\alpha}_i\frac{\partial F^i}{\partial y^{\alpha}}
+\left(-\frac{\partial \Lambda^b_i}{\partial x^a}p^j_b-\frac{\partial \Lambda^{\beta}_i}{\partial x^a}p^j_{\beta}+\Gamma^j_{ik}p^k_a-\Gamma^k_{ik}p^j_a\right)\frac{\partial F^i}{\partial p^j_a}\right.\\
&+\left.\left(-\frac{\partial \Lambda^b_i}{\partial y^{\alpha}}p^j_b-\frac{\partial \Lambda^{\beta}_i}{\partial y^{\alpha}}p^j_{\beta}+\Gamma^j_{ik}p^k_{\alpha}-\Gamma^k_{ik}p^j_{\alpha}\right)\frac{\partial F^i}{\partial p^j_{\alpha}}\right)\vol
\end{align*}
Since $F$ is Poisson, $F^i$ depends on $p^j_a$ or $p^j_{\alpha}$ only when $i=j$. In addition, as $\Lambda$ is $G$-invariant, $\partial \Lambda^a_i/\partial y^{\alpha}=0$ for $i=1,\dots,n$; $a=1,\dots,m-k-n $; $\alpha=1\dots k$. Thus,
\begin{align*}
d^hF=&\left(\frac{\partial F^i}{\partial x^i}+\Lambda^a_i\frac{\partial F^i}{\partial x^a}+\Lambda^{\alpha}_i\frac{\partial F^i}{\partial y^{\alpha}}
+\left(-\frac{\partial \Lambda^b_i}{\partial x^a}p^j_b-\frac{\partial \Lambda^{\beta}_i}{\partial x^a}p^j_{\beta}\right)\frac{\partial F^i}{\partial p^j_a}-\frac{\partial \Lambda^{\beta}_i}{\partial y^{\alpha}}p^j_{\beta}\frac{\partial F^i}{\partial p^j_{\alpha}}\right)\vol.
\end{align*}
Similarly, as $\Lambda$ is $G$-invariant, Lemma \ref{func.inv} can be applied to any $\partial \Lambda^{\gamma}_i/\partial y^{\alpha}$ for $i=1,\dots,n$; $\alpha,\gamma=1\dots k$. Furthermore as $F$ is $G$-invariant, using the derivative of $\kappa$ in Proposition \ref{dif.kappa} we obtain
\begin{align*}
d^hF=&\kappa^*\left(\frac{\partial f^i}{\partial x^i}+\mu^j_{\gamma}\frac{\partial A^{\gamma}_b}{\partial x^i}\frac{\partial f^i}{\partial \sigma^j_b}+\Lambda^a_i\frac{\partial f^i}{\partial x^a}+\mu^j_{\gamma}\frac{\partial A^{\gamma}_b}{\partial x^a}\Lambda^a_i\frac{\partial f^i}{\partial \sigma^j_b}\right.\\
&-\frac{1}{2}\mu^j_{\gamma}c^{\gamma}_{\beta\alpha}A^{\beta}_b\Lambda^{\alpha}_i\frac{\partial f^i}{\partial \sigma^j_{b}}-\frac{1}{2}\mu^j_{\gamma}c^{\gamma}_{\beta\alpha}\Lambda^{\alpha}_i\frac{\partial f^i}{\partial \mu^j_{\beta}}\\
&\left.+\left(-\frac{\partial \Lambda^b_i}{\partial x^a}\sigma^j_b+\frac{\partial \Lambda^b_i}{\partial x^a}\mu^j_{\gamma}A^{\gamma}_b-\frac{\partial \Lambda^{\beta}_i}{\partial x^a}\mu^j_{\beta}\right)\frac{\partial f^i}{\partial \sigma^j_a}+\frac{1}{2}\mu^j_{\beta}c^{\beta}_{\gamma\alpha}\Lambda^{\gamma}_i\left(\frac{\partial f^i}{\partial \mu^j_{\alpha}}+A^{\alpha}_b\frac{\partial f^i}{\partial \sigma^j_b}\right)\right)\vol.
\end{align*}
Regrouping for different partial derivatives of $f^i$,
\begin{align*}
d^hF=&\kappa^*\left(\frac{\partial f^i}{\partial x^i}+\Lambda^a_i\frac{\partial f^i}{\partial x^a}+\mu^j_{\gamma}c^{\gamma}_{\beta\alpha}\Lambda^{\beta}_i\frac{\partial f^i}{\partial \mu^j_{\alpha}}\right.\\
&\left.+\left(-\frac{\partial \Lambda^b_i}{\partial x^a}\sigma^j_b+\mu^j_{\gamma}\left(\frac{\partial A^{\gamma}_a}{\partial x^i}+\frac{\partial A^{\gamma}_a}{\partial x^b}\Lambda^b_i+\frac{\partial \Lambda^b_i}{\partial x^a}
A^{\gamma}_b-\frac{\partial \Lambda^{\gamma}_i}{\partial x^a}-c^{\gamma}_{\beta\alpha}A^{\beta}_a\Lambda^{\alpha}_i\right)\right)\frac{\partial f^i}{\partial \sigma^j_a}\right)\vol\\
=&\kappa^*(d^hf)+\kappa^*\left(\mu^i_{\gamma}\left(\frac{\partial A^{\gamma}_a}{\partial x^i}+\frac{\partial A^{\gamma}_a}{\partial x^b}\Lambda^b_i+\frac{\partial \Lambda^b_i}{\partial x^a}
A^{\gamma}_b-\frac{\partial \Lambda^{\gamma}_i}{\partial x^a}-c^{\gamma}_{\beta\alpha}A^{\beta}_a\Lambda^{\alpha}_i\right)Y^a\right)\vol
\end{align*}
where the specific structure of Poisson $(n-1)$-form $f$ has been used. To finish the proof, it suffices to show that the last term of the later expression is a curvature term;
\begin{align*}
&\frac{\partial A^{\gamma}_a}{\partial x^i}+\frac{\partial A^{\gamma}_a}{\partial x^b}\Lambda^b_i+\frac{\partial \Lambda^b_i}{\partial x^a}
A^{\gamma}_b-\frac{\partial \Lambda^{\gamma}_i}{\partial x^a}-c^{\gamma}_{\beta\alpha}A^{\beta}_a\Lambda^{\alpha}_i\\
&=\frac{\partial A^{\gamma}_a}{\partial x^i}+\frac{\partial A^{\gamma}_a}{\partial x^b}\Lambda^b_i-\frac{\partial A^{\gamma}_b}{\partial x^a}\Lambda^b_i+\frac{\partial (\Lambda^b_iA^{\gamma}_b)}{\partial x^a}
-\frac{\partial \Lambda^{\gamma}_i}{\partial x^a}-c^{\gamma}_{\beta\alpha}A^{\beta}_a\Lambda^{\alpha}_i\\
&=\frac{\partial A^{\gamma}_a}{\partial x^i}-\frac{\partial A^{\gamma}_i}{\partial x^a}+\Lambda^b_i\left(\frac{\partial A^{\gamma}_a}{\partial x^b}-\frac{\partial A^{\gamma}_b}{\partial x^a}\right)+\frac{\partial (A^{\gamma}_i+\Lambda^b_iA^{\gamma}_b-\Lambda^{\gamma}_i)}{\partial x^a}
-c^{\gamma}_{\beta\alpha}A^{\beta}_a\Lambda^{\alpha}_i\\
&=\frac{\partial A^{\gamma}_a}{\partial x^i}-\frac{\partial A^{\gamma}_i}{\partial x^a}-c^{\gamma}_{\beta\alpha}A^{\beta}_aA^{\alpha}_i+\Lambda^b_i\left(\frac{\partial A^{\gamma}_a}{\partial x^b}-\frac{\partial A^{\gamma}_b}{\partial x^a}-c^{\gamma}_{\beta\alpha}A^{\beta}_aA^{\alpha}_b\right)\\
&\phantom{=}+\frac{\partial (A^{\gamma}_i+\Lambda^b_iA^{\gamma}_b-\Lambda^{\gamma}_i)}{\partial x^a}+c^{\gamma}_{\beta\alpha}A^{\beta}_a\left(A^{\alpha}_i +A^{\alpha}_b\Lambda^b_i-\Lambda^{\alpha}_i\right).
\end{align*}
From the compatibility equation \eqref{compatibleP}, the last two terms vanish. Finally, the other two terms are a local expression of the curvature $\B$ applied to $\partial/\partial x^a$ and the lift of $\partial/\partial x^i$ from $M$ to $s\in\Sigma$.
\end{proof}
\begin{theorem} \label{reduction.intrinsic.eqs}
%Let $\pi_{MP}:P\to M$ be a bundle $\dots$  ¿Es necesario?
For any section $p$ of $\Pi_P\to M$, let $\sigma\oplus\mu=\kappa \circ p$ be the reduced section of
$$\Pi_P/G\cong\Pi_\Sigma\oplus(TM\otimes\g^*\otimes\bigwedge^{n-1}T^*M)\to\Sigma.$$
Then, the following are equivalent:
\begin{enumerate}[(i)]
\item for every Poisson $(n-1)$-form $F$ on $\Pi_P$, the following identity holds true:
\begin{equation}\label{intrinsic.eq}
\{F,H\}\vol=d(F\circ p)-d^{h}F\circ p.
\end{equation}
\item for every affine Poisson $(n-1)$-form $f$ on $\Pi_\Sigma\oplus(TM\otimes\g^*\otimes\bigwedge^{n-1}T^*M)$:
\begin{equation}\label{intrinsic.red.eq}
\{f,h\}\vol=d(f\circ(\sigma\oplus\mu))-d^{h}f\circ(\sigma\oplus\mu
)+\left\langle\mu, \bar{\B}(Y,\bar{\Lambda})\right\rangle\vol,
\end{equation}
where the bracket is defined by equation \eqref{reduced bracket}.
\end{enumerate}
\end{theorem}
\begin{proof}
From Theorem \ref{kappa.is.poisson}, $\kappa$ is Poisson and $\{F,H\}=\kappa^*\{f,h\}$. Furthermore, from Lemma \ref{red.dh}, $d^hF=\kappa^*\left(d^hf-\left\langle\mu, \bar{\B}(Y,\bar{\Lambda})\right\rangle\right)$. Since every affine Poisson $(n-1)$-form $f$ can be obtained by reduction of a Poisson $(n-1)$-form $F$ on $\Pi_P$, we only need to prove that  $d(F\circ p)=d(f\circ(\sigma\oplus\mu))$.
As $F$ is horizontal, for any $v_1,\dots v_{n-1}$ vectors in $T_xM$
\begin{align*}
(p^*F)_x(v_1,\dots,v_{n-1})&=F_{p(x)}(T_xp(v_1),\dots,T_xp(v_{n-1}))=F_{p(x)}(v_1,\dots,v_{n-1})\\&=(F\circ p)_x(v_1,\dots,v_{n-1}).
\end{align*}
In an analogous way, $(\sigma\oplus\mu)^*f=f\circ(\sigma\oplus\mu)$. Consequently,
$$F\circ p=p^*F=p^*\kappa^*f=(\sigma\oplus\mu)^*f=f\circ(\sigma\oplus\mu),$$
coincide as $(n-1)$-forms on $M$ and $d(F\circ p)=d(f\circ(\sigma\oplus\mu))$.
\end{proof}

Let $E\to P$ be a vector bundle over a fiber bundle $P\to M$, and $\nabla$ an affine connection. The divergence $\mathrm{div}^{\nabla}\chi\in \Gamma(M,E^*)$ of a section $\chi\in\Gamma(M,TM\otimes E^*)$ satisfies
\begin{align*}
\mathrm{div}\langle\xi,\eta\rangle=\langle\mathrm{div}^{\nabla}\chi,\eta\rangle+\langle\chi,\tilde{\nabla}\eta\rangle,
\end{align*} for any $\eta\in\Gamma(M,E)$,
where $\mathrm{div}$ is the usual divergence of a vector field on $M$ (with respect to the volume form $\vol$). In particular, $\mathrm{div}^{\A}=\mathrm{div}^{\nabla^{\A}}$ is defined from $\Ad(P)\to \Sigma$ and the affine connection $\nabla^{\A}$ induced by the principal connection $\A$ on $P\to\Sigma$.
\begin{theorem} \label{th.intrinsic}
In the same setting as in Theorem \ref{reduction.intrinsic.eqs}, the following statement is equivalent to (i) and (ii):
\begin{itemize}[(iii)]
\item the section $\sigma\oplus\mu$ satisfies the equations:
\begin{align}
\label{eq:ver.intrinsic}
\mathrm{div}^{\A} \mu &=\ad^*_{\frac{\delta h}{\delta \mu}}\mu,\\
\label{eq:hor.intrinsic}
\{\theta_Y,h\}_{\Sigma}\vol &= d(\theta_Y\circ(\sigma\oplus\mu))-d^{h}\theta_Y\circ(\sigma\oplus\mu)+\left\langle\mu, \bar{\B}\left(Y,\bar{\Lambda}+\frac{\delta h}{\delta\sigma}\right)\right\rangle\vol,
\end{align}
for every $\theta_Y$ Poisson $(n-1)$-form on $\Pi_\Sigma\oplus(TM\otimes\g^*\otimes\bigwedge^{n-1}T^*M)$ defined by $Y\in\Gamma(V\Sigma\to M)$ as in Theorem \ref{red.n-1.forms}.
\end{itemize}
\end{theorem}
\begin{proof}
Let $f$ be an affine Poisson $(n-1)$-form described by $Y\oplus \bar{\xi}$ as in Theorem \ref{red.n-1.forms}. In the case that $\bar{\xi}=0$, $f=\theta_Y$ and equation \eqref{intrinsic.red.eq} in Theorem \ref{reduction.intrinsic.eqs} takes the form
\begin{equation}
\{\theta_Y,h\}_{\Sigma}\vol-\left\langle\mu,\bar{\B}\left(Y,\frac{\delta h}{\delta\sigma}\right)\right\rangle\vol = d(\theta_Y\circ(\sigma\oplus\mu))-d^{h}\theta_Y\circ(\sigma\oplus\mu)+\left\langle\mu, \bar{\B}(Y,\bar{\Lambda})\right\rangle\vol,
\end{equation}
from which equation \eqref{eq:hor.intrinsic} is obtained. In the case $Y=0$, $f=\xi^{\alpha}\mu^i_{\alpha}\vol_i$ and equation \eqref{intrinsic.red.eq} has the local expression
\begin{multline*}
\left(\frac{\partial \xi^{\alpha}}{\partial x^a}\mu^i_{\alpha}+\mu^i_{\gamma}c^{\gamma}_{\beta\alpha}A^{\beta}_a\xi^{\alpha}\right)\frac{\partial h}{\partial\sigma^i_a}+\mu^i_{\gamma}c^{\gamma}_{\beta\alpha}\xi^{\alpha}\frac{\partial h}{\partial \mu^i_{\beta}}=\\=\frac{\partial\xi^{\alpha}}{\partial x^i}\mu^i_{\alpha}+\frac{\partial \xi^{\alpha}}{\partial x^a}\mu^i_{\alpha}\frac{\partial x^a}{\partial x^i}+\xi^{\alpha}\frac{\partial\mu^i_{\alpha}}{\partial x^i}-\left(\frac{\partial\xi^{\alpha}}{\partial x^i}\mu^i_{\alpha}+\frac{\partial \xi^{\alpha}}{\partial x^a}\mu^i_{\alpha}\Lambda^a_i+\xi^{\alpha}\mu^i_{\gamma}c^{\gamma}_{\beta\alpha}\Lambda^{\beta}_i\right).
\end{multline*}
As $\bar{\xi}\in\Gamma(\Ad P)$ is arbitrary, the terms proportional to $\xi^{\alpha}$ provide
\begin{equation} \label{loc.eq.ver.1}
\mu^i_{\gamma}c^{\gamma}_{\beta\alpha}\frac{\partial h}{\partial \mu^i_{\beta}}+\mu^i_{\gamma}c^{\gamma}_{\beta\alpha}A^{\beta}_a\frac{\partial h}{\partial \sigma^i_a}=\frac{\partial \mu^i_{\alpha}}{\partial x^i}-\mu^i_{\gamma}c^{\gamma}_{\beta\alpha}\Lambda^{\beta}_i,
\end{equation}
while the terms proportional to $\partial\xi^{\alpha}/\partial x^a$ provide
\begin{equation}
\label{loc.eq.holo.1}
\frac{\partial h}{\partial \sigma^i_a}=\frac{\partial x^a}{\partial x^i}-\Lambda^a_i.
\end{equation}
Substitution of equation \eqref{loc.eq.holo.1} into equation \eqref{loc.eq.ver.1} gives
\begin{equation*}\label{loc.eq.ver.bis}
\mu^i_{\gamma}c^{\gamma}_{\beta\alpha}\frac{\partial h}{\partial \mu^i_{\beta}}=\frac{\partial \mu^i_{\alpha}}{\partial x^i}-\mu^i_{\gamma}c^{\gamma}_{\beta\alpha}A^{\beta}_i-\mu^i_{\gamma}c^{\gamma}_{\beta\alpha}A^{\beta}_b\frac{\partial x^b}{\partial x^i},
\end{equation*}
which is the local expression of equation \eqref{eq:ver.intrinsic}. To this point, we have proved that equation \eqref{intrinsic.red.eq} in statement (ii) is equivalent to the system of equations defined by \eqref{eq:ver.intrinsic}, \eqref{eq:hor.intrinsic} and \eqref{loc.eq.holo.1}. To end the proof, we must show that \eqref{loc.eq.holo.1} is a consequence of equation \eqref{eq:hor.intrinsic}. Indeed, equation \eqref{eq:hor.intrinsic} is written in local coordinates from
\begin{equation*}
\{\theta_Y,h\}_{\Sigma}=\frac{\partial Y^b}{\partial x^a}\sigma^i_b\frac{\partial h}{\partial \sigma^i_a}-\left(\frac{\partial h}{\partial x^a}+\mu^i_{\gamma}c^{\gamma}_{\beta\alpha}A^{\beta}_a\frac{\partial h}{\partial \mu^i_{\alpha}}\right)Y^a,
\end{equation*}
\begin{equation*}
d(\theta_Y\circ(\sigma\oplus\mu))=\sigma^i_a\frac{\partial Y^a}{\partial x^i}+\frac{\partial x^b}{\partial x^i}\frac{\partial Y^a}{\partial x^b}\sigma^i_a+Y^a\frac{\partial\sigma^i_a}{\partial x^i},
\end{equation*}
and
\begin{equation*}
d^{h}\theta_Y\circ(\sigma\oplus\mu)=\sigma^i_a\frac{\partial Y^a}{\partial x^i}+\Lambda^b_i\sigma^i_a\frac{\partial Y^a}{\partial x^b}-\frac{\partial\Lambda^b_i}{\partial x^a}\sigma^i_bY^a.
\end{equation*} While the curvature term has the following local expression:
\begin{equation*}
\bar{\B}\left(Y,\cdot\right)= -Y^a\left(\frac{\partial A^{\gamma}_a}{\partial x^i}-\frac{\partial A^{\gamma}_i}{\partial x^a}+c^{\gamma}_{\beta\alpha}A^{\beta}_{i}A^{\alpha}_a\right)dx^i-Y^a\left(\frac{\partial A^{\gamma}_a}{\partial x^b}-\frac{\partial A^{\gamma}_b}{\partial x^a}+c^{\gamma}_{\beta\alpha}A^{\beta}_{b}A^{\alpha}_a\right)dx^b.
\end{equation*}
As $Y\in\mathfrak{X}(\Sigma)$ is arbitrary, $Y^a$ and $\partial Y^b/\partial x^a$ in the local expression of equation \eqref{eq:hor.intrinsic} are arbitrary. As before, equation \eqref{loc.eq.holo.1} easily follows from the terms proportional to $\partial Y^b/\partial x^a$ while the terms proportional to $Y^a$ provide:
\begin{multline*}
-\left(\frac{\partial h}{\partial x^a}+\mu^i_{\gamma}c^{\gamma}_{\beta\alpha}A^{\beta}_a\frac{\partial h}{\partial\mu^i_{\alpha}}\right)=\frac{\partial \sigma^i_a}{\partial x^i}+\frac{\partial \Lambda^b_i}{\partial x^a}\sigma^i_b\nonumber\\
-\mu^i_{\gamma}\left(\frac{\partial A^{\gamma}_a}{\partial x^i}-\frac{\partial A^{\gamma}_i}{\partial x^a}+c^{\gamma}_{\beta\alpha}A^{\beta}_{i}A^{\alpha}_a+\left(\frac{\partial h }{\partial \sigma^i_b}+\Lambda^b_i\right)\left(\frac{\partial A^{\gamma}_a}{\partial x^b}-\frac{\partial A^{\gamma}_b}{\partial x^a}+c^{\gamma}_{\beta\alpha}A^{\beta}_{b}A^{\alpha}_a\right)\right).
\end{multline*}
\end{proof}

\begin{remark}
When $P\to M$ is a principal bundle and the reduction is performed by its structure group $G$, we have $\Sigma = M$, and equation \eqref{eq:hor.intrinsic} is trivial, that is, the only reduced equation is \eqref{eq:ver.intrinsic}. This is the case studied in \cite{someremarks} and \eqref{eq:ver.intrinsic} is the so-called Lie-Poisson equation.

With respect to equation \eqref{eq:hor.intrinsic}, we could regard it as the bracket formulation of the Hamilton equations on the polysymplectic bundle of $\Sigma \to M$, now defined by a new bracket $\{,\}_{\Sigma}$ together with an additional curvature term. Both perturbations collect the fact that the Hamiltonian is actually defined on a product bundle  $\Pi_\Sigma\oplus(TM\otimes\g^*\otimes\bigwedge^{n-1}T^*M)$.

In short, Theorem \ref{th.intrinsic} provides the bracket formulation of the horizontal and vertical equations of a reduced Hamiltonian system in analogy with the vertical and horizontal Lagrange--Poincar\'e equations obtained in Lagrangian reduction \cite{EllisLP}. We explore this in the next section.
\end{remark}

\begin{remark} \label{local PP eqs} From the proof of Theorem \ref{th.intrinsic}, the local expressions of the Poisson-Poincar\'e equations are
%\begin{align}
%\label{loc.eq.ver}
%\mu^i_{\gamma}c^{\gamma}_{\beta\alpha}\frac{\partial h}{\partial \mu^i_{\beta}}&=\frac{\partial \mu^i_{\alpha}}{\partial x^i}-\mu^i_{\gamma}c^{\gamma}_{\beta\alpha}A^{\beta}_i-\mu^i_{\gamma}c^{\gamma}_{\beta\alpha}A^{\beta}_b\frac{\partial x^b}{\partial x^i},\\
%\label{loc.eq.holo}
%\frac{\partial h}{\partial \sigma^i_a}&=\frac{\partial x^a}{\partial x^i}-\Lambda^a_i, \\
%-\left(\frac{\partial h}{\partial x^a}+\mu^i_{\gamma}c^{\gamma}_{\beta\alpha}A^{\beta}_a\frac{\partial h}{\partial\mu^i_{\alpha}}\right)&=\frac{\partial \sigma^i_a}{\partial x^i}+\frac{\partial \Lambda^b_i}{\partial x^a}\sigma^i_b\nonumber\\
%\label{loc.eq.hor}-\mu^i_{\gamma}\left(\frac{\partial A^{\gamma}_a}{\partial x^i}-\frac{\partial A^{\gamma}_i}{\partial x^a}+\right.&\left.c^{\gamma}_{\beta\alpha}A^{\beta}_{a}A^{\alpha}_i+\frac{\partial x^b}{\partial x^i}\left(\frac{\partial A^{\gamma}_a}{\partial x^b}-\frac{\partial A^{\gamma}_b}{\partial x^a}+c^{\gamma}_{\beta\alpha}A^{\beta}_{a}A^{\alpha}_b\right)\right).
%\end{align}
\begin{align}
\label{loc.eq.ver}
\frac{\partial \mu^i_{\alpha}}{\partial x^i}-\mu^i_{\gamma}c^{\gamma}_{\beta\alpha}A^{\beta}_i-\mu^i_{\gamma}c^{\gamma}_{\beta\alpha}A^{\beta}_b\frac{\partial x^b}{\partial x^i}&=\mu^i_{\gamma}c^{\gamma}_{\beta\alpha}\frac{\partial h}{\partial \mu^i_{\beta}},\\
\label{loc.eq.holo}
\frac{\partial h}{\partial \sigma^i_a}&=\frac{\partial x^a}{\partial x^i}-\Lambda^a_i, \\
-\left(\frac{\partial h}{\partial x^a}+\mu^i_{\gamma}c^{\gamma}_{\beta\alpha}A^{\beta}_a\frac{\partial h}{\partial\mu^i_{\beta}}\right)&=\frac{\partial \sigma^i_a}{\partial x^i}+\frac{\partial \Lambda^b_i}{\partial x^a}\sigma^i_b\nonumber\\
\label{loc.eq.hor}-\mu^i_{\gamma}\left(\frac{\partial A^{\gamma}_a}{\partial x^i}-\frac{\partial A^{\gamma}_i}{\partial x^a}+c^{\gamma}_{\beta\alpha}\right.&\left.A^{\beta}_{a}A^{\alpha}_i+\frac{\partial x^b}{\partial x^i}\left(\frac{\partial A^{\gamma}_a}{\partial x^b}-\frac{\partial A^{\gamma}_b}{\partial x^a}+c^{\gamma}_{\beta\alpha}A^{\beta}_{a}A^{\alpha}_b\right)\right).
\end{align}
\end{remark}

%\begin{remark} ¿¿¿??? Let $s$ be the section of $\Sigma\to M$ obtained from the projection of $\sigma\oplus\mu$. equations \eqref{loc.eq.holo} and \eqref{loc.eq.hor} are equivalent to
%\begin{equation}
%\{f,h\}_{\Sigma}\vol = d(f\circ(\sigma\oplus\mu))-d^{h}f\circ(\sigma\oplus\mu)-\left\langle\mu, \bar{\B}(Y,ds)\right\rangle\vol,
%\end{equation}
%for any horizontal Poisson $(n-1)$-form on $\Pi_{\Sigma}\oplus(TM\otimes\g^*\otimes\bigwedge^{n}T^*M )$.
%\end{remark}

\section{Legendre Transformation} \label{SEC: Legendre}
The covariant Lagrange--Poincar\'e reduction studied in \cite{cov.red,EllisLP} is the Lagrangian picture of the Poisson--Poincar\'e reduction developed in this paper. Let $\mathcal{L}:J^1P\to\bigwedge^{n}T^*M$ be an hyper-regular Lagrangian density and $L:J^1P\to\R$ such that $\mathcal{L}=L\vol$, the \em covariant Legendre transformation, \em $\F\mathcal{L}:J^1P\to J^1P^*$ defined as the first-order vertical Taylor approximation to $\mathcal{L}$,
\begin{equation}
\F\mathcal{L}(j^1s)(j^1s')=\mathcal{L}(j^1s)+\frac{d}{d\varepsilon}\bigg|_{\varepsilon =0}\mathcal{L}(j^1s+\varepsilon(j^1s'-j^1s))
\end{equation}
is not an isomorphism as $\dim(J^1P^*)=\dim(J^1P)+1$. Yet, the linear Legendre transformation
\begin{equation}
\widehat{\F\mathcal{L}}(j^1s)(\vartheta)=\frac{d}{d\varepsilon}\bigg|_{\varepsilon =0}\mathcal{L}(j^1s+\varepsilon\vartheta),
\end{equation}
where $\vartheta \in T^*M\otimes VP$, defines a diffeomorphism between $J^1P$ and $\Pi_P.$ The Hamiltonian system $(\Pi_P,\F\mathcal{L}\circ\widehat{\F\mathcal{L}}^{-1})$ is equivalent to the Lagrangian system defined by $\mathcal{L}$. Provided an Ehresmann connection $\Lambda$ on $P\to M$, the system is described by the Hamiltonian
\begin{equation}\label{eq: equiv ham}
\mathcal{H}(p)=\left\langle p, \widehat{\F\mathcal{L}}^{-1}(p)-\Lambda\right\rangle-\mathcal{L}\circ \widehat{\F\mathcal{L}}^{-1}(p)
\end{equation}
which in local coordinates can be written as $\mathcal{H}=H\vol$ where $H=p^i_a(y^a_i-\Lambda^a_i)-L$. The map
\begin{align*}
\widehat{\F \mathcal{H}}:\Pi_P&\to T^*M\otimes VP \\
p &\mapsto \widehat{\F \mathcal{H}}(p)(\varrho)=\left.\frac{d}{d\varepsilon}\right\vert_{\varepsilon =0}\mathcal{H}(p+\varepsilon \varrho\vol)
\end{align*}
is called \em inverse Lagrange transformation \em since
\begin{center}
\begin{tikzpicture}[scale=1.5]
\node (A) at (0,0) {$T^*M\otimes VP$};
\node (B) at (0,1) {$J^1P$};
\node (C) at (1.5,1) {$\Pi_P$};
\draw[->,font=\scriptsize,>=angle 90]
(B) edge node[left]{$F_\Lambda$} (A);
\draw[->,font=\scriptsize,>=angle 90]
(B) edge node[above]{$\widehat{\F\mathcal{L}}$} (C);
\draw[->,font=\scriptsize,>=angle 90]
(C) edge node[right]{$\widehat{\F \mathcal{H}}$} (A);
\end{tikzpicture}
\end{center}
where $F_\Lambda$ is the linearization of the affine bundle $J^1P\to P$ taking $\Lambda$ as the zero section.

We shall see that Poisson--Poincar\'e reduction provides a reduced Legendre transformation. Suppose that $\mathcal{L}$ is $G$-invariant. On one hand, the linear Legendre transformation $\widehat{\F\mathcal{L}}:J^1P\to\Pi_P$ is $G$-equivariant: for every $g\in G$
\begin{equation*}
\widehat{\F\mathcal{L}}(gj^1s)(g\vartheta)=\left.\frac{d}{d\varepsilon}\right\vert_{\varepsilon =0}\mathcal{L}(gj^1s+\varepsilon(g\vartheta))=\left.\frac{d}{d\varepsilon}\right\vert_{\varepsilon =0}\mathcal{L}(j^1s+\varepsilon\vartheta)=\widehat{\F\mathcal{L}}(j^1s)(\vartheta).
\end{equation*}
Thus, there exists a reduced map $\big[\widehat{\F\mathcal{L}}\big]_G:J^1P/G\to\Pi_P/G$. On the other hand, the reduced Lagrangian $l$ defines the map
\begin{align*}
\widehat{\F l}: J^1P/G &\to \Pi_P/G \\
[j^1s] &\mapsto \widehat{\F l}\left([j^1s]\right)\left([\vartheta]\right)=\left.\frac{d}{d\varepsilon}\right\vert_{\varepsilon =0}l\left([j^1s]+\varepsilon [\vartheta]\right)\vol.
\end{align*}
Then, we have that $\widehat{\F l}=\big[\widehat{\F\mathcal{L}}\big]_G$ since
\begin{multline*}
\big[\widehat{\F\mathcal{L}}\big]_G\left([j^1s]\right)([\vartheta])=
\widehat{\F\mathcal{L}}(j^1s)(\vartheta)=\left.\frac{d}{d\varepsilon}\right\vert_{\varepsilon =0}\mathcal{L}\left(j^1s+\varepsilon \vartheta\right)\\=\left.\frac{d}{d\varepsilon}\right\vert_{\varepsilon =0}l\left([j^1s+\varepsilon\vartheta]\right)\vol=\left.\frac{d}{d\varepsilon}\right\vert_{\varepsilon =0}l\left([j^1s]+\varepsilon [\vartheta]\right)\vol=\widehat{\F l}\left([j^1s]\right)\left([\vartheta]\right).
\end{multline*}
Furthermore, the Hamiltonian in equation \eqref{eq: equiv ham} is $G$-invariant and there exists a reduced Hamiltonian $h:\Pi_P/G\to \bigwedge^nT^*M$. It is similarly proven that $\widehat{\F \mathcal{H}}$ is $G$-equivariant and $\widehat{\F h}=\big[\widehat{\F \mathcal{H}}\big]_G$.

Suppose that the Ehresmann connection $\Lambda$ is $G$-invariant, then there is a $G$-invariant section of $J^1P\to P$ which reduces to a section $[\Lambda]$ of $J^1P/G\to\Sigma$. Then, the linearization map $F_{\Lambda}$ is $G$-equivariant and $\left[F_{\Lambda}\right]_G=F_{[\Lambda]}$, the linearization of $J^1P/G$ taking $[\Lambda]$ as the zero section. In addition,
\begin{align}\label{redlegendreham}
h([p])&=H(p)=\left\langle [p],\big[\widehat{\F\mathcal{L}}^{-1}(p)-\Lambda\big]\right\rangle\vol^*-l\big[\widehat{\F\mathcal{L}}^{-1}(p)\big]\nonumber \\&
=\big\langle [p],\widehat{\F l}^{-1}([p])-[\Lambda]\big\rangle\vol^*-l\big(\widehat{\F l}^{-1}([p])\big)
\end{align}
and $h$ can be obtained from $\widehat{\F l}$ and $[\Lambda]$. In conclusion, the following diagram commutes
\begin{center}
\begin{tikzpicture}[scale=1.5]
\node (A) at (0,0) {$T^*M\otimes (VP)/G$};
\node (B) at (0,1.) {$J^1P/G$};
\node (C) at (1.5,1) {$\Pi_P/G$};
\draw[->,font=\scriptsize,>=angle 90]
(B) edge node[left]{$F_{[\Lambda]}$} (A);
\draw[->,font=\scriptsize,>=angle 90]
(B) edge node[above]{$\widehat{\F l}$} (C);
\draw[->,font=\scriptsize,>=angle 90]
(C) edge node[right]{$\widehat{\F h}$} (A);
\end{tikzpicture}
\end{center}
The map $\widehat{\F l}$ is called \em reduced Legendre transform \em and $\widehat{\F h}$ is called \em reduced reverse Legendre transform.\em

Given a compatible principal connection $\A$ on $P\to\Sigma$ in the sense of Definition \ref{compatible} and adapted local coordinates as in Section \ref{SEC: symmetries}, $\F l$ has the following local expression
\begin{equation*}
\F l\left(y^a_ i,\nu^{\alpha}_i\right)=\left(\sigma^i_a=\frac{\partial l}{\partial y^a_i}(y^a_ i,\nu^{\alpha}_i),\mu^i_{\alpha}=\frac{\partial l}{\partial \nu^{\alpha}_i}(y^a_ i,\nu^{\alpha}_i)\right),
\end{equation*}
and the reduced Hamiltonian in equation \eqref{redlegendreham} is
\begin{equation}\label{eq: redhamlegendre}
h(x^i,x^a,\sigma^i_a,\mu^i_{\alpha})=-l(x^i,x^a,y^a_i,\nu^{\alpha}_i)+\sigma^i_a(y^a_i-\Lambda^a_i)+\mu^{\alpha}_i\nu^{\alpha}_i,
\end{equation}
where $(y^a_ i,\nu^{\alpha}_i)=\widehat{\F l}^{-1}(\sigma^i_a,\mu^i_{\alpha})$. The Poisson--Poincar\'e equations \eqref{loc.eq.ver}, \eqref{loc.eq.holo} and \eqref{loc.eq.hor}  for $h$ in \eqref{eq: redhamlegendre} provide the Lagrange--Poincar\'e equations for $l$. More precisely, since equations \eqref{loc.eq.holo} for that Hamiltonian imposes that the solution of the Lagrangian system must be an holonomic section, the resulting equations can be interpreted as a field theoretic counterpart of the implicit Lagrange--Poincar\'e equations provided in \cite{DiracCotangent} in Mechanics. Obtaining these implicit Lagrange--Poincar\'e equations by reduction of systems as those introduced in \cite{HPmultiDirac} may prove to be an interesting future work.

\section{Example: SO(3)-strand} \label{SEC: example}
A $G$-strand is a map into a Lie group $G$, $g(t, s) :\R^2 \to G$, whose dynamics in $(t,s)\in\R^2$ may be obtained from Hamilton’s principle for a $G$-invariant Lagrangian. These systems and their integrability have been extensively studied in, for example, \cite{matrixgstrands, gstrands}. The equations of motion of these $G$-strands can be interpreted as the dynamics of a continuous spin chains. In the particular case where $G=SO(3)$, we may think of the strand as a continuous chain of rigid bodies (See \cite{MolStrand} for more details). We now present a Hamiltonian system corresponding to an $SO(3)$-strand in which the $SO(3)$ symmetry is broken in one specific direction $e_3$.

Consider the trivial bundle $P=\R^2 \times SO(3)$ over $\R^2$ with projection $(t,s,R)\mapsto x=(t,s)$ and the polysymplectic bundle
\begin{equation*}
\Pi_P=T\R^2\otimes V^*(\R^2 \times SO(3))\otimes\bigwedge^{2}T^*\R^2 = T\R^2\otimes T^*SO(3)\otimes\bigwedge^2T^*\R^2.
\end{equation*}
Any section of $\Pi_P\to\R^2$ can be written as
$$\frac{\partial}{\partial t} \otimes p^t(x) \otimes (dt\wedge ds)+\frac{\partial}{\partial s} \otimes p^s(x) \otimes (dt\wedge ds),$$
where $p^t$, $p^s:\R^2 \to T^*SO(3)$.

We write the Hamiltonian modeling the time evolution of a charged molecular or spin strand subject to an uniform electric field. We regard the strand as a chain of rigid bodies parametrized by $s$ and equipped with a dipole electric momentum. This Hamiltonian system is defined in $\Pi_P$ via the trivial Ehresmann connection $\Lambda:T\R^2\to TP$ on the product bundle $P=\R^2 \times SO(3)\to\R^2$ as
\begin{equation}\label{EX Hamiltonian}
H(R, p^t, p^s)=\frac{1}{2}\langle RI^{-1}R^{-1}p^t,p^t\rangle-\frac{1}{2}\langle RJ^{-1}R^{-1}p^s,p^s\rangle+e e_3\cdot R\chi,
\end{equation}
where $I$ is the inertia tensor of the rigid body, $J$ is a tensor describing the opposition to the rotation of consecutive rigid bodies of the strand, $\chi\in\R^3$ is the electric dipole of the rigid body configuration reference, and $\mathbf{E}=ee_3$ is a uniform electric field in the  direction $e_3$ of the spatial frame of reference. Since
\begin{equation}
\frac{\partial H}{\partial p^t}=RI^{-1}R^{-1}p^t,
\end{equation}
\begin{equation}
\frac{\partial H}{\partial p^s}=-RJ^{-1}R^{-1}p^s,
\end{equation}
and $\partial H/\partial R=ee_3\times R\chi$. The local Hamilton--Cartan equations for this system are:
\begin{equation}
R_t=RI^{-1}R^{-1}p^t,
\end{equation}
\begin{equation}
R_s=-RJ^{-1}R^{-1}p^s,
\end{equation}
\begin{equation}
ee_3\times R\chi=-\partial_tp^t-\partial_sp^s
\end{equation}
There is a left $SO(3)$ action on $P=\R^2\times SO(3)$ by left translations on the second factor which is lifted to $\Pi_P$. Since for any $Q\in SO(3)$,
\begin{align*}
H(Q&\cdot(R, p^t, p^s))=H(QR, Qp^t, Qp^s)=\\ &=\frac{1}{2}\langle QRI^{-1}R^{-1}Q^{-1}Qp^t,Qp^t\rangle-\frac{1}{2}\langle QRJ^{-1}R^{-1}Q^{-1}Qp^s,Qp^s\rangle+e e_3\cdot QR\chi\\&=\frac{1}{2}\langle RI^{-1}R^{-1}p^t,p^t\rangle-\frac{1}{2}\langle RJ^{-1}R^{-1}p^s,p^s\rangle+e Q^{-1}e_3\cdot R\chi,
\end{align*}
the Hamiltonian system is not $SO(3)$ invariant. Yet, it is invariant by the action of rotations around the axis $e_3$. In other words, the Hamiltonian system is $\sS^1$-invariant. There is a  principal bundle, $P\to P/\sS^1$, with projection
\begin{align*}
P=\R^2\times SO(3)&\to \Sigma=P/\sS^1=\R^2\times\sS^2\\
(t,s,R)&\mapsto  (t,s,\zeta=R^{-1}e_3).
\end{align*}
Let $v\in T_RSO(3)\simeq \mathbb{R}^3$, where the identification is given by the hat map. Then, $v$ can be decomposed in parallel and perpendicular components with respect to $e_3$ as $v=\langle v,e_3\rangle e_3-e_3\times(e_3\times v).$ It can easily be obtained that $T\pi_{P,\Sigma}(v)=R^{-1}(e_3\times v)$. Hence, the mechanical connection is given by $\A(v)=\langle v, e_3\rangle$ and for any $v_{\zeta}\in T_{\zeta}\sS^2$ its horizontal lift to $R$ is $v_{\zeta}^{h,R}=-e_3\times R v_{\zeta}=R(-\zeta\times v_{\zeta})$.

This connection allows to identify $\Pi_P/\sS^1$ with $\Pi_\Sigma\oplus( T\R^2\otimes \R\otimes\bigwedge^2T^*\R^2)=\Pi_\Sigma\oplus T^*\R^2$ as in Proposition \ref{quotient.ident}. We now obtain the reduced multimomenta with the projection defined in that Proposition. For $\alpha\in T^*_RSO(3)$ and $v_{\zeta}\in T_{\zeta}\sS^2$,
\begin{equation*}
\langle \hlift^*_{\A}(\alpha),v_{\zeta}\rangle =\langle\alpha,R(-\zeta\times v_{\zeta})\rangle=
\langle  v_{\zeta} ,\zeta \times R^{-1}\alpha\rangle.
\end{equation*}
Thus, $\hlift^*_{\A}(\alpha)=\zeta \times R^{-1}\alpha$. Let $a\in \R$, an element in the Lie algebra of $\sS^1$. Then,
\begin{equation*}
\langle J(\alpha),a\rangle =\langle\alpha,a_R\rangle= a \langle \alpha,(e_3)_R \rangle = a \langle R^{-1}\alpha, (R^{-1}e_3)_e \rangle= a \langle R^{-1}\alpha, \zeta \rangle.
\end{equation*}
Thus, $J(\alpha)=\langle R^{-1}\alpha, \zeta \rangle$. In conclusion, the multimomenta $p^t, p^s$ in $\Pi_P$ reduce to
\begin{equation}
\sigma^t=\zeta\times R^{-1}p^t, \hspace{5mm }\sigma^s=\zeta\times R^{-1}p^s,
\end{equation}
\begin{equation}
\mu^t=\langle\zeta, R^{-1}p^t\rangle, \hspace{5mm } \mu^s=\langle\zeta, R^{-1}p^s\rangle.
\end{equation}
Observe that for $i=t,s$, $R^{-1}p^i$ can be decomposed in a parallel and perpendicular to $\zeta$ in terms of the reduced multimomenta $\sigma^i$, $\mu^i$. Indeed,
\begin{equation*}
R^{-1}p^i=\mu^i\zeta-\zeta\times \sigma^i,
\end{equation*}
and substitution into $H$ provides the reduced Hamiltonian
\begin{multline}
h(\zeta,\sigma^t,\sigma^s,\mu^t,\mu^s)=\frac{1}{2}(\mu^t)^2 \langle\zeta,I^{-1}\zeta\rangle+\mu^t\langle  \sigma^t,\zeta \times I^{-1}\zeta\rangle+ \frac{1}{2}\langle\zeta\times\sigma^t,I^{-1}(\zeta\times\sigma^t)\rangle\\
-\frac{1}{2}(\mu^s)^2 \langle\zeta,J^{-1}\zeta\rangle-\mu^s\langle  \sigma^s,\zeta \times J^{-1}\zeta\rangle- \frac{1}{2}\langle\zeta\times\sigma^s,J^{-1}(\zeta\times\sigma^t)\rangle+e\zeta\cdot\xi.
\end{multline}
The equations of motion are given by the local Poisson--Poincar\'e equations gathered in Remark \ref{local PP eqs};
{\fontsize{9 pt}{10pt}\selectfont
\begin{align}
\zeta_t&=\zeta\times I^{-1}(\mu^t\zeta-\zeta\times\sigma^t),\label{eq string holo t}\\
\zeta_s&=-\zeta\times J^{-1}(\mu^s\zeta-\zeta\times\sigma^s),\label{eq string holo s}\\
\partial_t\sigma^t+\partial_s\sigma^s&=-\zeta\times\sigma^t\langle\zeta,I^{-1}(\mu^t\zeta-\zeta\times\sigma^t)\rangle+\zeta\times\sigma^s\langle\zeta,J^{-1}(\mu^s\zeta-\zeta\times\sigma^s)\rangle-e\chi^{\parallel},\label{eq string hor}\\
\partial_t\mu^t+\partial_s\mu^s&=0, \label{eq string ver}
\end{align}}
where $\chi^{\parallel}=\chi-\zeta\langle\xi,\zeta\rangle$ is the ortogonal projection of $\chi$ to $T_{\zeta}\sS^2$. Equations \eqref{eq string holo t} and \eqref{eq string holo s} are the particularization of equation \eqref{loc.eq.holo}, and equation \eqref{eq string ver} is the vertical equation obtained from \eqref{loc.eq.ver}. The horizontal equation \eqref{eq string hor} is obtained from \eqref{loc.eq.hor} using that $\B(\zeta_t,\cdot)=-\zeta\times\zeta_t.$ In the special case $I=J=\mathrm{Id}_{\R^3}$ and $\chi=0$, equations (\ref{eq string holo t}-\ref{eq string ver}) are the decomposition in terms parallel and perpendicular to $\zeta$ of the equations of motion of the classical chiral model for the $SO(3)$-strand and the equations for harmonic maps from $\R^2$ to $SO(3)$ studied in \cite{Uhlenbeck}.

This Hamiltonian system can be formulated from the Lagrangian point of view via the hyperregular $\sS^1$-invariant Lagrangian
\begin{equation}
L(R,R_t,R_s)=\frac{1}{2}\langle R^{-1}R_t,IR^{-1}R_t\rangle-\frac{1}{2}\langle R^{-1}R_s,JR^{-1}R_s\rangle-ee_3\cdot R\chi
\end{equation}
defined in $J^1(\R^2\times SO(3))$. The multimomenta
\begin{equation}
p^t=\frac{\partial L}{\partial R_t}= RIR^{-1}R_t, \hspace{3mm} p^s=\frac{\partial L}{\partial R_s}=-RJR^{-1}R_s,
\end{equation}
and the Hamiltonian in equation \eqref{EX Hamiltonian} can be obtained from the Legendre transform
$H=\langle p^t, R_t\rangle +\langle p^s, R_s\rangle -L.$ The $\sS^1$-invariant Lagrangian can be reduced to a Lagrangian $l$ on $J^1(\R^2\times\sS^2)\oplus(T^*\R^2\otimes(\R^2\times\sS^2\times \R))$ as explained in \cite{cov.red, EllisLP}. Indeed,
\begin{align*}
l(\zeta,\zeta_t,\zeta_s,\eta,\xi)=&\frac{1}{2}\eta^2 \langle \zeta, I\zeta\rangle+\eta\langle \zeta, I(\zeta_t\times\zeta)\rangle +\frac{1}{2} \langle \zeta_t\times\zeta, I(\zeta_t\times\zeta)\rangle \\
&-\frac{1}{2}\xi^2 \langle \zeta, J\zeta\rangle-\xi\langle \zeta, J(\zeta_s\times\zeta)\rangle -\frac{1}{2} \langle \zeta_s\times\zeta, J(\zeta_s\times\zeta)\rangle
-e\zeta\cdot\chi,
\end{align*}
where $\eta=\A(R_t)$, $\xi=\A(R_s)$. In turn, the momenta $\sigma^t, \sigma^s, \mu^t, \mu^s$ can be obtained from vertical derivatives of $l$,
\begin{equation}
\sigma^t=\frac{\partial l}{\partial \zeta_t}= \zeta\times (\eta I\zeta+I(\zeta_t\times\zeta) ), \hspace{3mm} \sigma^s=\frac{\partial l}{\partial \zeta_s}=-\zeta\times (\xi J\zeta+J(\zeta_s\times\zeta) ),
\end{equation}
\begin{equation}
\mu^t=\frac{\partial l}{\partial \eta}= \langle\zeta,\eta I\zeta+I(\zeta_t\times\zeta)\rangle, \hspace{3mm} \mu^s=\frac{\partial l}{\partial \xi}=-\langle\zeta,\xi J\zeta+J(\zeta_s\times\zeta)\rangle,
\end{equation}
and the reduced Hamiltonian $h$ can be deduced from the Legendre transform $h=\langle\sigma^t,\zeta_t\rangle+\langle\sigma^s,\zeta_s\rangle+\mu^t\eta+\mu^s\xi-l$.

\section*{Acknowledgements}
Both authors have been partially supported by Agencia Nacional de Investigaci\'on, Spain, under grants no. PGC2018-098321-B-I00 and PID2021-126124NB-I00. MAB has been partially supported by Ministerio de Universidades (Spain) under grant no. FPU17-02701.
\bibliography{references}{}
\bibliographystyle{abbrvbf}
%abbrv
\end{document}